\documentclass{amsart}

\usepackage{amssymb}
\usepackage{amsrefs}
\usepackage{amsmath}
\usepackage{amsfonts}
\usepackage{amsthm}
\usepackage{array}
\usepackage{bm}
\usepackage[shortlabels]{enumitem}
\usepackage{xcolor}
\usepackage{eucal}
\usepackage{upgreek}

\usepackage{hyperref}

\pdfoutput=1

\usepackage{tikz-cd}
\usepackage{tkz-euclide} % no need to load TikZ
\usepackage{pgfplots}
\hypersetup{
colorlinks,
linkcolor={red!50!black},
citecolor={blue!50!black},
urlcolor={blue!80!black}
}

\newtheorem{thm}{Theorem}[section]
\newtheorem{prop}[thm]{Proposition}
\newtheorem{rmk}[thm]{Remark}
\newtheorem{lem}[thm]{Lemma}

\newtheorem*{prob*}{Problem}

\newtheorem*{cor*}{Corollary}
\newtheorem*{thm*}{Theorem}

\theoremstyle{definition}

\theoremstyle{plain} % just in case the style had changed
\newcommand{\thistheoremname}{}
\newtheorem{genericthm}[thm]{\thistheoremname}

\newtheorem*{genericthm*}{\thistheoremname}
\newenvironment{namedthm*}[1]
{\renewcommand{\thistheoremname}{#1}%
\begin{genericthm*}}
{\end{genericthm*}}

\theoremstyle{remark}

\numberwithin{equation}{section}

\newcommand{\R}{\mathbb{R}}
\newcommand{\C}{\mathbb{C}}

\newcommand{\SO}{\mathrm{SO}}
\newcommand{\SU}{\mathrm{SU}}

\newcommand{\Sp}{\mathrm{Sp}}
\newcommand{\G}{\mathrm{G}}
\newcommand{\Hg}{\mathrm{H}}
\newcommand{\U}{\mathrm{U}}

\newcommand{\Spin}{\mathrm{Spin}}
\newcommand{\Ker}{\mathrm{Ker}}

\newcommand{\Hom}{\mathrm{Hom}}

\newcommand{\Ad}{\mathrm{Ad}}

\newcommand{\Sym}{\mathrm{Sym}}

\newcommand{\Z}{\mathbb{Z}}

\newcommand{\RP}{\mathbb{RP}}

\newcommand{\vol}{\mathrm{vol}}

\newcommand{\SFF}{\mathrm{I\!I}}

\usepackage{scalerel}
\pdfoutput=1

\begin{document}
	
	\title{The Morse Index of Quartic Minimal Hypersurfaces}
	
	\author{Gavin Ball}
	\address{\textsc{University of Wisconsin-Madison},
		\textsc{Department of Mathematics},
		\textsc{Madison, WI, USA}}
	\email{gball3@wisc.edu}
	\urladdr{https://www.gavincfball.com/}
	
	\author{Jesse Madnick}
	\address{\textsc{University of Oregon},
		\textsc{Department of Mathematics},
		\textsc{Eugene, WA, USA}}
	\email{jmadnick@uoregon.edu}
	\urladdr{https://sites.google.com/view/jesse-madnick}
	
	\author{Uwe Semmelmann}
	\address{\textsc{Institut f\"ur Geometrie und Topologie},
		\textsc{Fachbereich Mathematik},
		\textsc{Universit\"at Stuttgart}}
	\email{uwe.semmelmann@mathematik.uni-stuttgart.de}
	
	\begin{abstract} The homogeneous minimal hypersurfaces in $S^n$ have $g \in \{1,2,3,4,6\}$ distinct (constant) principal curvatures.  While the Morse index and nullity have been calculated for all such hypersurfaces having $g = 1,2,3$, it has remained an open problem to compute these quantities for any of those with $g = 4$ or $6$.
		
	In this paper, we calculate the Morse index and nullity of two homogeneous minimal hypersurfaces in $S^n$ with $g = 4$.  Moreover, we observe that their Laplace spectra contain irrational eigenvalues that are not expressible in radicals.
	\end{abstract}
	
	\maketitle

        \tableofcontents
	
	\section{Introduction}
	
	\subsection{Background and results}
	
	A minimal hypersurface $M^{n-1} \subset S^n$ is a critical point of the area functional: for every variation $(M_t)$ of $M$, the function $A(t) = \mathrm{Area}(M_t)$ satisfies $A'(0) = 0$.  The \emph{Morse index} and \emph{nullity} of $M$, denoted $\mathrm{Ind}(M)$ and $\mathrm{Nul}(M)$, are the number of infinitesimal variations of $M$ that satisfy $A''(0) < 0$ and $A''(0) = 0$, respectively.  Simons \cite{Simons68} observed that $\mathrm{Ind}(M) \geq 1$ and $\mathrm{Nul}(M) \geq n$, with equality holding in either case if and only if $M$ is a totally geodesic $S^{n-1}$. \\
	\indent Given a minimal hypersurface $M \subset S^n$ with unit normal vector field $\nu$, along with a variation $\iota_t \colon M_t \to S^n$ having variation vector field $\left.\frac{d}{dt}\right|_{t = 0} \iota_t = f\nu$ with $f \in C^\infty(M)$, the second variation formula states that
	\begin{equation*}
		A''(0) = \int_M f\,\mathcal{J}_M(f)\,\mathrm{vol}_M
	\end{equation*}
	where $\mathcal{J}_M(f) = -(\Delta_M + |\SFF|^2 + (n-1))f$ is the \emph{Jacobi operator} of $M$.  Here, $\Delta_M$ is the (analyst's) Laplacian of $M$, and $\SFF$ is the second fundamental form of $M$.  As a strongly elliptic operator, $\mathcal{J}_M$ may be diagonalized: its spectrum is an increasing sequence of real numbers $\mu_1 < \mu_2 < \cdots < \mu_s = 0 < \mu_{s+1} < \cdots \to \infty$ of finite multiplicities $\ell_1, \ell_2, \ldots$.  From this point of view, the Morse index and nullity are given by
	\begin{align*}
		\mathrm{Ind}(M) & = \ell_1 + \cdots + \ell_{s-1} & \mathrm{Nul}(M) & = \ell_s
	\end{align*}
	\indent The Morse index in particular is a fundamental invariant of minimal hypersurfaces, and is the subject of various conjectures.  Unfortunately, computing $\mathrm{Ind}(M)$ and $\mathrm{Nul}(M)$ is often impossible in actual practice.  As such, there are presently very few minimal hypersurfaces whose Morse index and nullity are known explicitly.  The ones that we do know help guide conjectures, placing bounds on what is reasonable to expect.  For example, it is known that every non-equatorial minimal hypersurface satisfies $\mathrm{Ind}(M) \geq n+2$, and it is conjectured that equality holds precisely for the Clifford hypersurfaces $S^p(\sqrt{ \frac{p}{n-1} }) \times S^q(\sqrt{ \frac{q}{n-1} })$ with $p+q = n-1$: see, e.g., \cites{Urbano90,Perdomo01}.\\
	
	\indent In view of this difficulty, it is natural to restrict attention to special classes of minimal hypersurfaces.  Perhaps the most natural class consists of the \emph{homogeneous} ones, those given as orbits of Lie groups.  In that case, $M$ has $g \in \{1,2,3,4,6\}$ distinct principal curvatures, the cone $\mathrm{C}(M) = \{rx \in \R^{n+1} \colon r > 0, x \in M\}$ is an algebraic variety of degree $g$ \cites{Munz1,Munz2}, and the Jacobi operator reduces to
	$$\mathcal{J}_M = -(\Delta_M + g(n-1)),$$
	% where $g \in \{1,2,3,4,6\}$ is the number of distinct principal curvatures of $M$.
	Thus, for these hypersurfaces, the spectrum of $\mathcal{J}_M$ is a shift of that of $\Delta_M$. However, in general, computing the spectrum of $\Delta_M$ is itself highly non-trivial.  Indeed, whereas the spectrum of a Riemannian homogeneous space $\G/\mathrm{H}$ can in principle be computed when the metric is normal, the situation for non-normal metrics is considerably more difficult.   (The metrics induced on homogeneous minimal hypersurfaces of spheres are, in general, not normal.) \\
	\indent The homogeneous minimal hypersurfaces in $S^n$ were classified by Hsiang and Lawson \cite{HsiangLawson71}, falling into $5$ infinite families and $9$ sporadic exceptions: see Figure \ref{fig-HomogenousMinimal}.  The Morse index and nullity have been computed for all of those having $g = 1$, $2$ and $3$ principal curvatures: the case of $g = 1$ is due to Simons \cite{Simons68}, and the $g = 3$ case to Solomon \cites{solomon90i,solomon90ii}.  By contrast, calculation of the Morse index and nullity for all of the examples with $g = 4$ or $6$ has remained an open problem.  Our main results make progress in the $g=4$ case:

	\begin{thm}  Let $M^8 \subset S^9$ denote the $\SO(5)$-invariant minimal isoparametric hypersurface.  Then $M$ has Morse index $275$ and nullity $35$.  In particular, every Jacobi field on $M$ is rotational (i.e., generated by the $\SO(10)$-action).
	\end{thm}
	
	\begin{thm} Let $M^4 \subset S^5$ denote the $(\SO(3) \times \SO(2))$-invariant minimal isoparametric hypersurface.  Then $M$ has Morse index $39$ and nullity $11$.  In particular, every Jacobi field on $M$ is rotational (i.e., generated by the $\SO(6)$-action).
	\end{thm}
	
	\begin{rmk} That every Jacobi field on $M^{n-1} \subset S^n$ is rotational has consequences for minimal desingularizations of its cone $\mathrm{C}(M) = \{rx \in \R^{n+1} \colon r > 0, x \in M\}$.  Indeed, a result of Allard and Almgren \cite{AllAlm81} implies that such desingularizations are normal graphs over $\mathrm{C}(M)$.
	\end{rmk}
	
	\indent Now, in the case where $M \subset S^n$ is a homogeneous minimal hypersurface with $g = 3$, Solomon \cites{solomon90i,solomon90ii} computes the Morse index and nullity by deriving an explicit formula for the Laplace eigenvalues and their multiplicities. In particular, such spectra consist of \emph{integers}.  For our examples, we obtain an algorithm that computes the eigenvalues $\lambda$ of $M$ satisfying $\lambda \geq -K$ for any desired $K > 0$.  As a consequence, we find by stark contrast that minimal hypersurfaces with $g = 4$ can admit \emph{irrational} Laplace eigenvalues.  In fact, more is true:
	
	\begin{thm} Let $M$ denote the $\SO(5)$-invariant minimal hypersurface in $S^9$ or the $(\SO(3) \times \SO(2))$-invariant minimal hypersurface in $S^5$.  Then the Laplace spectrum of $M$ contains eigenvalues not expressible in radicals.
	\end{thm}
	
	\indent An important point of view on homogeneous minimal hypersurfaces $M \subset S^n$ is that they are \emph{isoparametric}, i.e., have constant principal curvatures.  There is a vast body of literature on this subject, stretching back to Cartan in the 1930's \cites{Cartan38,Cartan39}: see \cite{Chi20} for a survey.  In this context, Tang and Yan \cite{tangyan13} proved that the first Laplace eigenvalue of $M$ satisfies $\lambda_1(M) = -(n-1)$, thereby verifying Yau's conjecture in the isoparametric case.  In addition, Solomon \cite{solomon92} studied the  spectrum of isoparametric minimal hypersurfaces with $g = 4$ distinct principal curvatures, proving several results.  In particular, he showed that $-2(n-1)$ is always a Laplace eigenvalue.

	\begin{figure}[h]
		\small
		$$\begin{tabular}{| c | c | c | c | c |}     \hline 
			$g$ & $\dim(M)$ & $M = \G/\mathrm{H}$ & $\mathrm{Ind}(M)$ & $\mathrm{Nul}(M)$  \\ \hline \hline
			$1$ & $n-1$ & $S^{n-1}$ & $1$ & $n$  \\ \hline
			$2$ & $n-1$ & $S^p \times S^{n-1-p}$ & $n+2$ & $ (p+1)(n-p)$  \\ \hline
			$3$ & $3$ & \rule{0pt}{4.2ex} $\displaystyle \frac{\SO(3)}{\Z_2 \times \Z_2}$ & $20$ & $7$  \\ 
			$3$ & $6$ & \rule{0pt}{4.2ex} $\displaystyle \frac{\SU(3)}{T^2}$ & $44$ & $20$  \\ 
			$3$ & $12$ & \rule{0pt}{4.2ex} $\displaystyle \frac{\Sp(3)}{\Sp(1)^3}$ & $119$ & $70$ \\ 
			$3$ & $24$ & \rule{0pt}{4.2ex} $\displaystyle \frac{\mathrm{F}_4}{\Spin(8)}$ & $377$ & $273$  \\ \hline
			$4$ & $2k-2$ \ $(k \geq 3$) & \rule{0pt}{4.2ex} $\displaystyle \frac{\SO(k) \times \SO(2)}{\SO(k-2) \times \Z_2}$ & $39 \text{ if } k = 3$ & $11 \text{ if } k = 3$  \\ 
			$4$ & $4k-2$ \ $(k \geq 2$) & \rule{0pt}{4.2ex} $\displaystyle \frac{\mathrm{S}(\U(k) \times \U(2))}{\SU(k-2) \times T^2}$ &  &     \\ 
			$4$ & $8k-2$ \ $(k \geq 2$) & \rule{0pt}{4.2ex} $\displaystyle \frac{\Sp(k) \times \Sp(2)}{\Sp(k-2) \times \Sp(1)^2}$ &  &    \\ 
			$4$ & $8$ & \rule{0pt}{4.2ex} $\displaystyle \frac{\SO(5)}{T^2}$ & $275$ & $35$   \\ 
			$4$ & $18$ & \rule{0pt}{4.2ex} $\displaystyle \frac{\U(5)}{\SU(2) \times \SU(2) \times \U(1)}$ &  &     \\ 
			$4$ & $30$ & \rule{0pt}{4.2ex} $\displaystyle \frac{\Spin(10) \times \U(1)}{\SU(4) \times \U(1)}$ &  &    \\ \hline
			$6$ & $6$ & \rule{0pt}{4.2ex} $\displaystyle \frac{\SO(4)}{\Z_2 \times \Z_2}$ & &     \\ 
			$6$ & $12$ & \rule{0pt}{4.2ex} $\displaystyle \frac{\G_2}{T^2}$ &  &    \\ \hline
		\end{tabular}$$
		\caption{Homogeneous Minimal Hypersurfaces in $S^n$}
		\label{fig-HomogenousMinimal}
	\end{figure}
	\normalsize

	\subsection{Method}
	
	\indent We now provide a brief outline of our method for computing the Morse index and nullity of a homogeneous minimal hypersurface $M^{n-1}$ in $S^n$.  It is a refinement and extension of a technique used by Solomon \cites{solomon90i,solomon90ii}, who applied it to the case of cubic ($g = 3$) minimal isoparametric hypersurfaces.  Roughly speaking, the idea is to decompose $L^2(M)$ into pieces, and then compute the Laplace spectrum on each piece. \\
	\indent To be more precise, let $M^{n-1} = \G/\mathrm{H}$ be an abstract homogeneous Riemannian manifold, and fix a positive constant $K > 0$.  The following four-step process yields an algorithm for calculating the Laplace eigenvalues $\lambda$ of $M$ (and their multiplicities) satisfying $\lambda \geq -K$.
	\begin{enumerate}
		\item  Decompose $L^2(M) = \bigoplus_\rho W_\rho$ into a direct sum of isotypical summands $W_{\rho}$.  Since the Laplacian is $\G$-invariant, the restrictions $\left.\Delta\right|_{W_\rho} \colon W_\rho \to W_\rho$ are well-defined. \\
		\item  For each summand $W_\rho$, find an explicit ``generating basis” $\mathcal{B}_\rho$.  (The finite set of functions $\mathcal{B}_\rho$ is typically a basis of a particular subspace of $W_\rho$, not of $W_\rho$ itself.) \\
		\item For each $F \in \mathcal{B}_\rho$, express $\Delta F$ as a linear combination of functions in $\mathcal{B}_\rho$.  Because $\mathcal{B}_\rho$ generates $W_\rho$ in a suitable sense, this is sufficient to compute the spectrum of $\left.\Delta\right|_{W_\rho}$. \\
		\item Apply the eigenvalue bound of Theorem \ref{thm:MainBound}. \\
	\end{enumerate}
	In $\S$\ref{subsec:Method}, we explain these steps in greater detail.  Note that the presence of an embedding $M^{n-1} \to S^n$ is not strictly necessary, but we find that it is quite helpful in constructing the ``generating basis" $\mathcal{B}_\rho$ (Step 2).
	
	\subsection{Outlook} It is plausible that an application of the above method could yield the Morse index and nullity of the three families of quartic minimal hypersurfaces in $S^{2k-1}$, $S^{4k-1}$, and $S^{8k-1}$.  On the other hand, we expect that the situation for the two sporadic quartic hypersurfaces of dimensions $18$ and $30$, as well as that for the two sextic hypersurfaces, is likely to be computationally challenging. \\
	\indent More generally, we expect that this technique may be applied to compute the spectrum of other non-normal Riemannian homogeneous spaces. \\
	
	\noindent \textbf{Conventions:}
	\begin{itemize}
		\item The Laplacian $\Delta := \Delta_M$ of a Riemannian manifold $M$ has non-positive spectrum.
		\item Via the Theorem of Highest Weight, we occasionally identify irreducible $\G$-representations with their highest weights.
		\item We sum over repeated indices unless otherwise indicated.
		\item Various computations in this work were performed with the aid of a computer algebra system (viz., \textsc{Maple}).
	\end{itemize}
	
	\noindent \textbf{Acknowledgments:} We thank Bruce Solomon and Christos Mantoulidis for helpful conversations.

	\section{Preliminaries} \label{sec:Method}

	\indent In $\S$\ref{subsec:Method}, we explain in detail our method for computing the Morse index and nullity of homogeneous minimal hypersurfaces in $S^n$.  It is a refinement and extension of a technique used by Solomon \cites{solomon90i,solomon90ii}, who applied it to the cubic case ($g = 3$).  Then, in $\S$\ref{subsec:EigenvalueBounds}, we state and prove the eigenvalue bound (Theorem \ref{thm:MainBound}) required for the last step in the process.  Finally, in $\S$\ref{subsec:MovingFrames}, we set up our moving frame apparatus for homogeneous hypersurfaces in $S^n$, which will enable us to efficiently compute geometric quantities in concrete examples.  We will employ basic concepts from the representation theory of Lie groups and Lie algebras; see the Appendix for details.
	
	\subsection{The method} \label{subsec:Method}
	
	\indent Let $M = \G/\mathrm{H}$ be a homogeneous space with $\G$ compact and $\mathrm{H} \leq \G$ closed, and let $g_M$ be a $\G$-invariant metric on $M$.  The $\G$-action on $M = \G/\mathrm{H}$ given by left-multiplication induces a $\G$-action on $L^2(M;\C)$ via the formula $(g^\sharp f)(x) := f(g^{-1}x)$, where $f \in L^2(M;\C)$ and $x \in M$.  The $\G$-module $L^2(M;\C)$ then decomposes into $\G$-submodules in the following way:
	
	\begin{prop}[Frobenius Reciprocity] \label{prop:FrobRecGen} Let $\widehat{\G}$ denote the set of equivalence classes of finite-dimensional, irreducible complex $\G$-modules.  There exists a $\G$-module isomorphism
		$$L^2(M; \C) \cong \bigoplus_{\rho \in \widehat{\G}} V_{\rho} \otimes \mathrm{Hom}_{\mathrm{H}}(V_\rho, \C) 
		= \bigoplus_{\rho \in \widehat{\G}} V_{\rho}^{\oplus m(\rho)}$$
		where the \emph{multiplicity of $V_\rho$} is
		$$m(\rho) = \dim_{\C}  \mathrm{Hom}_{\mathrm{H}}(V_{\rho}, \C) = \dim_{\C}\{v \in V_\rho \colon hv = v, \ \forall h \in \mathrm{H}\}$$
	\end{prop}
	\noindent We refer to each $V_\rho^{\oplus m(\rho)} = V_\rho \oplus \cdots \oplus V_\rho$ as an \emph{isotypical summand}.  The method now proceeds as follows.  \\

	\textbf{Step 1.} Compute $m(\rho)$.  In general, this can be non-trivial. \\
	
	\textbf{Step 2a.} Find explicit $\G$-equivariant functions $\mathbf{z}_i \colon M \to V_{\rho_i}$.  Here, the $\G$-modules $V_{\rho_1}, \ldots, V_{\rho_\ell}$ need not be distinct.  The functions $\mathbf{z}_1, \ldots, \mathbf{z}_\ell$ should be sufficiently numerous so as to ``generate" $L^2(M; \C)$ in a sense to be described in Step 2b. \\
	
	\textbf{Step 2b.} For each irreducible representation $V_\rho$, its highest weight subspace $L_\rho \subset V_\rho$ has complex dimension $1$.  After choosing a non-zero highest weight vector, we may identify $L_\rho$ with $\C$.  Having done this, let $\pi_\rho \colon V_{\rho} \to \C$ denote the corresponding projection map onto $L_\rho \simeq \C$, and consider the \emph{highest weight functions}
	\begin{align*}
		\zeta_i := \pi_{\rho_i} \circ \mathbf{z}_i \colon M \to \C.
	\end{align*}
	We say that $\zeta_i$ has  \emph{weight} $\rho_i$. Crucial properties of such functions are: 
	
	\begin{prop}[\cite{solomon90i}] \label{prop:SpanHighestWeight} ${}$
		\begin{enumerate}[(a)]
			\item The set of highest weight functions with weight $\rho$ is a $\C$-vector space of complex dimension $m(\rho)$.
			\item If $\zeta_1, \ldots, \zeta_\ell$ are highest weight functions of weights $\rho_1, \ldots, \rho_\ell$, respectively, then $\zeta_1^{k_1} \cdots \zeta_{\ell}^{k_\ell}$ is a highest weight function of weight $k_1 \rho_1 + \cdots + k_\ell \rho_\ell$.
			\item If $\zeta$ is a highest weight function of weight $\rho$, then $\mathrm{span}_{\C}\{g^\sharp \zeta \colon g \in \G\} \cong V_\rho$.
			\item If $\zeta_1$, $\zeta_2$ are highest weight functions of weights $\rho_1$ and $\rho_2$, then  $\nabla \zeta_1 \cdot \nabla \zeta_2$ and $\Delta \zeta_i$ are highest weight functions of weights $\rho_1 + \rho_2$ and $\rho_i$, respectively.
		\end{enumerate}
	\end{prop}
	
	For each isotypical summand $V_{\rho}^{\oplus m(\rho)}$, construct a basis $\mathcal{B}_\rho$ of the set of highest weight functions of weight $\rho$.  In practice, this amounts to finding a finite set $\mathcal{B}_\rho$ of monomials $\zeta_1^{k_1} \cdots \zeta_\ell^{k_\ell}$ that satisfies the following properties:
	\begin{enumerate}
		\item Each $F \in \mathcal{B}_\rho$ has weight $\rho$.
		\item $\mathcal{B}_\rho$ is linearly independent in the $\C$-vector space $L^2(M;\C)$.
		\item $\dim_{\C}(\mathrm{span}_{\C}(\mathcal{B}_\rho)) = m(\rho)$.
	\end{enumerate}
	Having done this, Proposition \ref{prop:SpanHighestWeight}(c) now implies that
	\begin{equation} \label{eq:IsotypicalID}
		V_\rho^{\oplus m(\rho)} \cong \mathrm{span}_{\C}\{ g^\sharp \mathcal{B}_\rho \colon g \in \mathrm{G}\}.
	\end{equation}
	
	\textbf{Step 3.} For each isotypical summand $V_\rho^{\oplus m(\rho)}$, compute the matrix representation of the Laplacian  $\Delta|_{V_\rho^{\oplus m(\rho)}} \colon V_\rho^{\oplus m(\rho)} \to V_\rho^{\oplus m(\rho)}$.  To do this, note that Proposition \ref{prop:SpanHighestWeight}(d) implies $\Delta(\operatorname{span}(\mathcal{B}_\rho)) \subset \operatorname{span}(\mathcal{B}_\rho)$.  Therefore, in view of (\ref{eq:IsotypicalID}) and the $\G$-invariance $\Delta(g^\sharp F) = \Delta F$, it suffices to compute $\Delta$ on the $m(\rho)$-dimensional subspace $\mathrm{span}_{\C}(\mathcal{B}_\rho) \subset \mathrm{span}_{\C}\{ g^\sharp \mathcal{B}_\rho \colon g \in \G\}$. 
	Computing $\Delta F$ for each $F \in \mathcal{B}_\rho$, and expressing the result as a linear combination of functions in $\mathcal{B}_\rho$, yields a matrix representation of $\Delta$ on the subspace $\mathrm{span}_{\C}(\mathcal{B}_\rho)$. \\
	\indent Now, by the Laplacian product rule, calculating $\Delta F = \Delta(\zeta_1^{k_1} \cdots \zeta_\ell^{k_\ell})$ requires the calculation of the Laplacians $\Delta \zeta_i$ and the pairwise dot products $\nabla \zeta_i \cdot \nabla \zeta_j$.  The difficulty of this computation depends on the complexity of $M$ and of the functions $\zeta_i$, which can in general have lengthy polynomial expressions. \\
	
	\indent \textbf{Step 4.} For each isotypical summand $V_\rho^{\oplus m(\rho)}$ of $L^2(M;\C)$, the result of Step 3 allows us to extract the eigenvalues of $\Delta|_{V_\rho^{\oplus m(\rho)}}$.  Each such eigenvalue has multiplicity $\dim_{\C}(V_\rho)$.  So, to compute the Morse index of $M^{n-1} \subset S^n$, we require an exhaustive list of those summands housing $-\Delta$ eigenvalues that are strictly less than $g(n-1)$. \\ 
	\indent Now, Theorem \ref{thm:MainBound} provides a lower bound on the $-\Delta$ eigenvalues on a given $V_\rho^{\oplus m(\rho)}$.  Therefore, if an isotypical summand $V_\rho^{\oplus m(\rho)}$ contributes to the Morse index, then its corresponding lower bound cannot exceed $g(n-1)$.  This criterion yields a finite list of candidate summands $V_\rho^{\oplus m(\rho)}$, thereby reducing the calculation of the index to a straightforward linear algebra computation.
		
	\subsection{Eigenvalue bounds} \label{subsec:EigenvalueBounds}
	
	Let $N = \G / \mathrm{H}$ be a homogeneous space with $\mathrm{G}$ compact. Let $g_1$ and $g_2$ be two $\G$-invariant metrics on $N$ and suppose $V$ is a finite-dimensional subspace of $L^2(N)$ invariant under both Laplacians $\Delta_{g_1}, \Delta_{g_2}.$ 
	
	Let $r_{\mathrm{min}}$ and $r_{\mathrm{max}}$ denote the minimum and maximum eigenvalues of $g^{-1}_1$ with respect to $g^{-1}_2,$ and let $\mu_{\mathrm{min}}$ and $ \mu_{\mathrm{max}}$ denote the minimum and maximum eigenvalues of $-\Delta_{g_2}$ on $V.$
	
	\begin{thm} \label{thm:MainBound}
		Let $\lambda$ be an eigenvalue of $-\Delta_{g_1}$ on $V$. We have the following bounds:
		\begin{equation*}
			r_{\mathrm{min}} \mu_{\mathrm{min}} \leq \lambda \leq r_{\mathrm{max}} \mu_{\mathrm{max}}
		\end{equation*}
	\end{thm}
	
	\begin{proof}
		The minimum and maximum eigenvalues of $-\Delta_{g_2}$ are characterized in terms of the Rayleigh quotient by
		\begin{equation*} % \label{eq:genrayleighcharac}
			\begin{aligned}
				\mu_{\mathrm{min}} &= \min_{f \in V} \frac{(-\Delta f, f)}{(f,f)}  = \min_{f \in V} \frac{\int_N g_2(df, df) \, \mathrm{vol}_{g_2}}{\int_N f^2 \,\mathrm{vol}_{g_2}}, \\
				\mu_{\mathrm{max}} &= \max_{f \in V} \frac{(-\Delta f, f)}{(f,f)}  = \max_{f \in V} \frac{\int_N g_2(df, df) \, \mathrm{vol}_{g_2}}{\int_N f^2 \,\mathrm{vol}_{g_2}}.
			\end{aligned}
		\end{equation*}
		Let $f$ be a function on $N$. We have
		\begin{equation*}
			r_{\mathrm{min}} g_{2}(df,df) \leq g_1(df, df) \leq r_{\mathrm{max}} g_{2}(df,df).
		\end{equation*}
		We may integrate this inequality on $M$ to obtain
		\begin{equation*}
			\begin{aligned}
				r_{\mathrm{min}} \mu_{\mathrm{min}} \leq & \, r_{\mathrm{min}} \frac{\int_N g_{2} (df, df) \, \mathrm{vol}_{g_2}}{\int_N f^2 \,\mathrm{vol}_{g_2}} \leq \frac{\int_N g_1(df, df) \, \mathrm{vol}_{g_2}}{\int_N f^2 \,\mathrm{vol}_{g_2}} \\
				& \leq r_{\mathrm{max}} \frac{\int_N g_2 (df, df) \, \mathrm{vol}_{g_2}}{\int_N f^2 \,\mathrm{vol}_{g_2}} \leq r_{\mathrm{max}} \mu_{\mathrm{max}}.
			\end{aligned}
		\end{equation*}
		for any function $f \in V.$ The volume forms $\mathrm{vol}_{g_1}$ and $\mathrm{vol}_{g_2}$ differ by a positive contant, so we have 
		\begin{equation*}
			\frac{\int_N g_1(df, df) \, \mathrm{vol}_{g_2}}{\int_N f^2 \,\mathrm{vol}_{g_2}} =  \frac{\int_N g_1(df, df) \, \mathrm{vol}_{g_1}}{\int_N f^2 \,\mathrm{vol}_{g_1}}.
		\end{equation*}
		To complete the proof, we note that any eigenvalue $\lambda$ of $-\Delta_{g_1}$ must lie in the set
		\begin{equation*}
			\left\lbrace \left. \frac{\int_N g_1(df, df) \, \mathrm{vol}_{g_1}}{\int_N f^2 \,\mathrm{vol}_{g_1}} \right\rvert f \in V \right\rbrace. \qedhere
		\end{equation*}
	\end{proof}
	
	\begin{rmk}
		In applications of this result, $g_2$ will be a normal homogeneous metric on $N$, and $V$ will be an isotypical component of $L^2(N)$ corresponding to some irreducible $\G$-representation $W$. The Laplacian of the normal metric on this isotypical component is equal to $-\mu(W) \mathrm{Id},$ where $\mu(W)$ is the eigenvalue of the Casimir operator of $\G$ on $W$. In particular, we have $\mu_{\mathrm{min}} = \mu(W) = \mu_{\mathrm{max}}.$
	\end{rmk}
	
	\subsection{Moving frames} \label{subsec:MovingFrames}
	
	\subsubsection{The ambient sphere}
	
	We will be interested in hypersurfaces in the unit $n$-sphere $S^n = \{ X \in \R^{n+1} \colon |X|^2 = 1\}$.  Note that $\SO(n+1)$ acts transitively on $S^n$ with stabilizer $\SO(n)$.  Let $\mathbf{g} \colon \SO(n+1) \to \mathrm{M}_{n+1}(\R)$ denote the obvious embedding into the set of $n \times n$ matrices, and write
	$$\mathbf{g} = \begin{bmatrix} \mathbf{x} & \mathbf{e}_1 & \cdots & \mathbf{e}_n \end{bmatrix}.$$
	We view the columns $\mathbf{x}, \mathbf{e}_1, \ldots, \mathbf{e}_n \colon \SO(n+1) \to S^n$ as vector-valued functions, and the map $\mathbf{x} \colon \SO(n+1) \to S^n$ as a principal $\SO(n)$-bundle.  In fact, $\SO(n+1) \cong F_{\SO(n)}S^n$ can be identified with the oriented orthonormal frame bundle of $S^n$.
	
	For computations, let $\eta \in \Omega^1(\SO(n+1); \mathfrak{so}(n+1))$ denote the Maurer-Cartan form, so that $\eta = \mathbf{g}^{-1}d\mathbf{g}$.  As a matrix-valued $1$-form, we write
	$$\eta = \begin{bmatrix} 0 & -\alpha_i \\ \alpha_i & \beta_{ij} \end{bmatrix}$$
	where $\beta_{ij} = -\beta_{ji}$.  Geometrically, the round metric $g_{\mathrm{rd}}$ on $S^n$ pulls back to 
	$$\mathbf{x}^*g_{\mathrm{rd}} = \alpha_1^2 + \cdots + \alpha_n^2,$$
	and its Levi-Civita connection form is the matrix-valued $1$-form $(\beta_{ij}) \in \Omega^1(F_{\SO(n)}S^n; \mathfrak{so}(n))$. Now, expanding the equation $d\mathbf{g} = \mathbf{g}\eta$, we obtain the \emph{structure equations}
	\begin{equation}
		\begin{aligned} \label{eq:StrEqns}
			d\mathbf{x} & = \alpha_i \mathbf{e}_i \\
			d\mathbf{e}_i & = -\alpha_i \mathbf{x} + \beta_{ij} \mathbf{e}_j.
		\end{aligned}
	\end{equation}

\subsubsection{Homogeneous hypersurfaces}

	Suppose a compact Lie group $\G$ acts on the $n$-sphere $S^n$ with cohomogeneity one via a linear isometric action $\rho \colon \G \to \SO(n+1)$. Let $\Hg$ denote the principal isotropy group and suppose the distance between the singular orbits is $r.$
	
	Suppose $h : [0,r] \to S^n$ is a geodesic in $S^n$ intersecting each $\G$-orbit exactly once and such that $h(0)$ and $h(r)$ lie in the singular $\G$-orbits. Let $\widetilde{h} : [0,r] \to \SO(n+1)$ be a lift of $h$ to $\SO(n+1)$ along the coset projection $\mathbf{x} : \SO(n+1) \to S^n = \SO(n+1) / \SO(n).$ 
	
	Consider the map $\psi \colon \G \times [0,r] \to \SO(n+1)$ given by
	\begin{equation*}
		\psi : (g, \theta) \mapsto \rho(g) \widetilde{h}( \theta).
	\end{equation*}
	For a fixed value of $\theta$, we obtain a map $\psi_\theta : \G \to \SO(n+1)$ defined by $\psi_\theta(g) = \psi(g, \theta).$ This map $\psi_\theta$ descends to a map $\varphi_\theta \colon \G / \Hg \to S^n$ given by $\varphi_\theta(gH) = \rho(g)(h(\theta))$.
	We obtain a commutative diagram
	\begin{equation*}
		\begin{tikzcd}
			{\mathrm{G}} & {\mathrm{SO}(n+1)} \\
			{\mathrm{G}/\mathrm{H}} & {S^n}
			\arrow["{\psi_\theta}", from=1-1, to=1-2]
			\arrow["\pi"', from=1-1, to=2-1]
			\arrow["\mathbf{x}", from=1-2, to=2-2]
			\arrow["{\varphi_\theta}"', from=2-1, to=2-2]
		\end{tikzcd}
	\end{equation*}
	The image $M_\theta := \varphi_\theta(\G / \Hg)$ for $\theta \in [0,r]$ is the $\G$-orbit of the element $h(\theta) \in S^n.$ In particular, if $\theta \in (0,r),$ then $\varphi_\theta$ is an embedding and $M_\theta$ is an isoparametric hypersurface in $S^n.$ The elements of $\psi_\theta(\G) \subset \SO(n+1)$ may be thought of as orthonormal frames for $S^n$ adapted to the intrinsic geometry of the hypersurface $\varphi_\theta(\G / \Hg)$, and $\psi_\theta(\G)$ itself is the principal $\Hg$-bundle over $\varphi_\theta(\G / \Hg)$ consisting of all such adapted $\Hg$-frames.

	Let $\omega \in \Omega^1(G; \mathfrak{g})$ denote the $\mathfrak{g}$-valued left-invariant Maurer-Cartan form on $\G$. Letting $\mathbf{f} \colon \G \to \mathrm{M}_{n+1}(\R)$ denote the embedding given by $\rho$, we have $\rho_*(\omega) = \mathbf{f}^{-1}d\mathbf{f}$. By expanding $\psi^*(\eta) = (\mathbf{f}\widetilde{h})^{-1} d(\mathbf{f}\widetilde{h}),$ we calculate
	\begin{equation*}
		\psi^* ( \eta ) = \widetilde{h}^{-1}(\theta) \rho_* (\omega) \widetilde{h}(\theta) + \widetilde{h}^{-1}(\theta) d (\widetilde{h}(\theta)).
	\end{equation*}
	Therefore, on the hypersurface $M_\theta$, we have
	\begin{equation} \label{eq:EtaAlphaBeta-Relation}
		\psi_\theta^* ( \eta ) = \widetilde{h}(\theta)^{-1} \rho_*(\omega) \widetilde{h}(\theta).
	\end{equation}
	Using this relation, we may write the components $\alpha_i$ and $\beta_{ij}$ of $\eta$ in terms of the Maurer-Cartan form of $\G.$ This allows us to easily compute the induced metric and second fundamental form of $M_\theta$ in examples.
	
	\subsubsection{Differentiation on $M_\theta$} \label{subsubsec:Differentiation}
	
	Let $Y_i,$ be a basis for the Lie algebra $\mathfrak{g}$ orthonormal for the Killing form and adapted so that $Y_1, \ldots, Y_{m}$ span $\mathfrak{h}$ and $Y_{m+1}, \ldots, Y_{m+n-1}$ span the isotropy representation $\mathfrak{m}$ of $\G / \Hg$. Using this basis, we may write the Maurer-Cartan form $\omega$ of $\G$ as $\omega = \sum \omega_i Y_i.$ 
	
	The technique of the previous section allows us to express the induced metric $g$ on $M_\theta$ in terms of the $\omega_i$ as
	\begin{equation*}
		\pi^*g = \sum_{i,j = m+1}^{n+m-1} S_{ij} \omega_i \omega_j
	\end{equation*}
	for some constant positive definite symmetric matrix $S_{ij}.$ We may diagonalize $S_{ij}$ to obtain an orthonormal coframe $\xi_1, \ldots, \xi_{n-1}$ for $\mathfrak{m}$ so that 
	\begin{equation*}
		\pi^*g = \sum_{i=1}^{n-1} \xi_i^2.
	\end{equation*}
	
	If $f$ is a function on $\G / \Hg,$ then we may pull back $f$ to $\G$ along the coset projection $\pi : \G \to \G / \Hg.$ We may then express $d (\pi^* f) = \pi^* (d f)$ in terms of the coframe $\xi_i$ as
	\begin{equation*}
		\pi^* df = \sum_{i=1}^{n-1} f_i \xi_i
	\end{equation*}
	for some functions $f_i$ on $\G.$ The functions $f_i$ for $ 1 \leq i \leq n-1$ may be thought of as the components of a section of the associated bundle $ \G \times_\Hg \mathfrak{m},$ i.e., the tangent bundle of $\G / \Hg.$ If $v_i$ are the vector fields on $\G$ dual to $\xi_i,$ then the vector field $\sum_{i=1}^{n-1} f_i v_i$ on $\G$ descends to give a vector field on $\G / \Hg$, namely the gradient $\nabla^g f.$  The Laplacian of $f$ is given by $\Delta f = *_g d *_g d f.$
	
	By combining the above discussion with the equations (\ref{eq:StrEqns}) and the expression for the 1-forms $\alpha_i,$ $\beta_{ij}$ in terms of $\omega_i,$ we may now compute the gradient and Laplacian of any function of the components of the vectors $\mathbf{x}$ and $\mathbf{e}_i.$
	
	\section{The $\SO(5)$-Invariant Minimal Hypersurface in $S^9$}
	
	\indent In this section, we implement the steps of $\S$\ref{subsec:Method} to compute the Laplace spectrum of the $\SO(5)$-invariant minimal hypersurface in $S^9$, and thereby calculate its Morse index and nullity.  We continue with the notation of $\S$\ref{subsec:MovingFrames}, taking $n = 9$, $\mathrm{G} = \SO(5)$, $\mathrm{H} = \mathrm{T}^2$, and $\rho = \mathrm{Ad} \colon \SO(5) \to \SO(10)$.  Moreover, $\eta$ and $\omega$ will denote the Maurer-Cartan forms on $\SO(10)$ and $\SO(5)$, respectively.

	\subsection{The Minimal Hypersurface $M \subset S^9$}
	
	\subsubsection{The $\SO(5)$-action on $S^9$}
	
	We will view the ambient space as $\R^{10} = \mathfrak{so}(5)$, or sometimes as $\R^{10} = \Lambda^2(\R^5)$.  We take the following basis of $\R^{10}$:
	\begin{align*}
		Y_1 & = E_{23} & Y_3 & = E_{12} & Y_7 & = \textstyle \frac{1}{\sqrt{2}}( E_{24} + E_{35} ) \\
		Y_2 & = E_{45} & Y_4 & = E_{13} & Y_8 & = \textstyle \frac{1}{\sqrt{2}}( E_{25} - E_{34} ) \\
		&  & Y_5 & = E_{14} & Y_9 & = \textstyle \frac{1}{\sqrt{2}}( E_{24} - E_{35} ) \\
		&  & Y_6 & = E_{15} & Y_{10} & = \textstyle \frac{1}{\sqrt{2}}( E_{25} + E_{34} ).
	\end{align*}
	Here, $E_{ij}$ is the matrix with $1$ in the $(i,j)$-entry, and $-1$ in the $(j,i)$-entry, and $0$'s elsewhere.  This basis of $\R^{10}$ is orthonormal with respect to the positive-definite inner product
	$$\langle X,Y \rangle = \textstyle -\frac{1}{6}\mathrm{tr}(\mathrm{ad}_X \circ \mathrm{ad}_Y) = -\frac{1}{2}\mathrm{tr}(XY).$$
	
	\indent Now, the Lie group $\G = \SO(5)$ acts on $\R^{10} = \mathfrak{so}(5)$ via the adjoint action, yielding an $\SO(5)$-action on the unit $9$-sphere $S^9 = \{ X \in \mathfrak{so}(5) \colon |X|^2 = 1\}$.  Let $\mathrm{T}^2 \leq \SO(5)$ the maximal torus whose Lie algebra is 
	$$\mathfrak{t}^2 = \mathrm{span}(Y_1, Y_2).$$
	Then the $\mathrm{T}^2$-action on $\mathfrak{so}(5)/\mathfrak{t}^2$ gives the $T^2$-invariant decomposition
	\begin{equation} \label{eq:T2Decomp}
		\frac{\mathfrak{so}(5)}{\mathfrak{t}^2} = (Y_3, Y_4)_{1,0} \oplus (Y_5, Y_6)_{0,1} \oplus (Y_7, Y_8)_{-1,1} \oplus (Y_9, Y_{10})_{1,1}
	\end{equation}
	where the subscripts denote the weights. It is well-known that the principal orbits of the $\SO(5)$-action on $S^9$ action have stabilizer $\Hg = \mathrm{T}^2$, and are therefore equivariantly diffeomorphic to $N := \SO(5)/\mathrm{T}^2$.
	
	\subsubsection{Invariant metrics on $N$}
	
	We pause to consider the abstract homogeneous space $N = \SO(5)/\mathrm{T}^2$.  As in $\S$\ref{subsec:MovingFrames}, we write the Maurer-Cartan form of $\SO(5)$ as $\omega = \sum \omega_i Y_i$ for $1$-forms $\omega_1, \ldots, \omega_{10} \in \Omega^1(\SO(5))$.  Then every $\SO(5)$-invariant metric on $N$ pulls back via the coset projection $\pi \colon \SO(5) \to N$ to
	$$\pi^*g_{a_1, a_2, a_3, a_4} = a_1(\omega_3^2 + \omega_4^2) + a_2(\omega_5^2 + \omega_6^2) + a_3(\omega_7^2 + \omega_8^2) + a_4(\omega_9^2 + \omega_{10}^2),$$
	where $a_i > 0$ are positive constants.  Note that the case $a_1 = a_2 = a_3 = a_4 = 1$ corresponds to the normal homogeneous metric.

	\subsubsection{The isoparametric hypersurfaces $M_\theta \subset S^9$}
	
	\indent We now study the principal $\SO(5)$-orbits in $S^9$.  By a theorem of Cartan, every element of $
	\R^{10} = \mathfrak{so}(5)$ is $\Ad_{\SO(5)}$-equivalent to an element in $\mathfrak{t}^2$.  In fact, the curve $h \colon [0, \frac{\pi}{4}] \to S^9$ given by
	$$h(\theta) = \cos(\theta) Y_1 + \sin(\theta)Y_2$$
	is a geodesic in $S^9$ that intersects each $\SO(5)$-orbit exactly once, and the principal orbits correspond to $\theta \in (0, \frac{\pi}{4})$.  For $\theta \in (0, \frac{\pi}{4})$, we let $\varphi_\theta \colon \SO(5)/\mathrm{T}^2 \to S^9$ denote the orbit map $\varphi_\theta(g\mathrm{T}^2) = \mathrm{Ad}_g(h(\theta))$, and let
	$$M_\theta := \varphi_\theta(N) \subset S^9$$
	denote the orbit of $h(\theta) \in S^9$. \\
	\indent To perform computations, we work on the level of Lie groups.  For this, we let $\widetilde{h} \colon [0, \frac{\pi}{4}] \to \SO(10)$ denote the following lift of $h$:
	$$\widetilde{h}(\theta) = \begin{pmatrix}
		\cos \theta & -\sin \theta & 0 \\
		\sin \theta & \cos \theta & 0 \\
		0 & 0 & \mathrm{Id}_8 \end{pmatrix}\!,$$
	and let $\psi_\theta \colon \SO(5) \to \SO(10)$ denote $\psi_\theta(g) = \mathrm{Ad}_g \circ \widetilde{h}(\theta)$, as in $\S$\ref{subsec:MovingFrames}.
	Altogether, we have a commutative diagram
	\begin{equation*}
		\begin{tikzcd}
			\SO(5) \arrow[r, "\psi_\theta"] \arrow[d, "\pi"'] & \SO(10) \arrow[d, "\mathbf{x}"] \\
			N \arrow[r, "\varphi_\theta"']  & S^9           
		\end{tikzcd}
	\end{equation*}
	Elements $(\mathbf{x}, \mathbf{e}_1, \ldots, \mathbf{e}_9)$ of $\SO(5) \leq \SO(10)$ may be viewed as adapted $\mathrm{T}^2$-frames on $M_\theta$.  Geometrically, at a point $p \in M_\theta$, we have that $\mathbf{x} = p$ is the position vector, $\mathbf{e}_1 = \mathbf{n}$ is the normal vector to $M_\theta \subset S^9$ at $p$, and in view of (\ref{eq:T2Decomp}), there is a decomposition
	$$T_pM_\theta = T_{1,0} \oplus T_{0,1} \oplus T_{1,-1} \oplus T_{1,1}$$
	where
	\begin{align*}
		T_{1,0} & = \mathrm{span}(\mathbf{e}_2, \mathbf{e}_3), & T_{0,1} & = \mathrm{span}(\mathbf{e}_4, \mathbf{e}_5), & T_{1,-1} & = \mathrm{span}(\mathbf{e}_6, \mathbf{e}_7), & T_{1,1} & = \mathrm{span}(\mathbf{e}_8, \mathbf{e}_9). \\
	\end{align*}
	
	\indent Working on $\SO(5)$ and suppressing pullbacks via $\psi_\theta \colon \SO(5) \to \SO(10)$ from the notation, an application of equation (\ref{eq:EtaAlphaBeta-Relation}) yields:
	\begin{align*}
		\alpha_1 & = 0 \\
		\begin{bmatrix} \alpha_2 \\ \alpha_3 \end{bmatrix} & = \cos \theta \begin{bmatrix} -\omega_4 \\ \omega_3 \end{bmatrix} & \begin{bmatrix} \alpha_6 \\ \alpha_7 \end{bmatrix} & = (\sin \theta - \cos \theta) \begin{bmatrix} -\omega_8 \\ \omega_7 \end{bmatrix} \\
		\begin{bmatrix} \alpha_4 \\ \alpha_5 \end{bmatrix} & = \sin \theta \begin{bmatrix} -\omega_6 \\ \omega_5 \end{bmatrix} & \begin{bmatrix} \alpha_8 \\ \alpha_9 \end{bmatrix} & = (\sin \theta + \cos \theta) \begin{bmatrix} -\omega_{10} \\ \omega_9 \end{bmatrix}
	\end{align*}
	Therefore, the induced metric on $M_\theta \subset S^9$ has
	\begin{align*}
		a_1 & = \cos^2\theta  & a_3 & = 1 - 2 \sin \theta \cos \theta \\
		a_2 & = \sin^2\theta & a_4 & = 1 + 2 \sin \theta \cos \theta.
	\end{align*}
	Moreover, the symmetric $2$-tensor on $\SO(5)$ given by $\SFF =  \alpha_2 \circ \beta_{12} +  \alpha_3 \circ \beta_{13} + \cdots \alpha_9 \circ \beta_{19}$
	descends to $N$ as the second fundamental form of $M_\theta \subset S^9$.  In fact, using (\ref{eq:EtaAlphaBeta-Relation}), we may calculate
	\begin{align*}
		\begin{bmatrix} \beta_{12} \\ \beta_{13} \end{bmatrix} & = \tan \theta \begin{bmatrix} \alpha_2 \\ \alpha_3 \end{bmatrix} & \begin{bmatrix} \beta_{16} \\ \beta_{17} \end{bmatrix} & = \frac{1 + \tan \theta}{1 - \tan \theta} \begin{bmatrix} \alpha_6 \\ \alpha_7 \end{bmatrix} \\
		\begin{bmatrix} \beta_{14} \\ \beta_{15} \end{bmatrix} & = -\cot \theta \begin{bmatrix} \alpha_4 \\ \alpha_5 \end{bmatrix} & \begin{bmatrix} \beta_{18} \\ \beta_{19} \end{bmatrix} & = \frac{\tan \theta - 1}{1 + \tan \theta} \begin{bmatrix} \alpha_8 \\ \alpha_9 \end{bmatrix}
	\end{align*}
	This verifies that $M_\theta \subset S^9$ is isoparametric with four distinct principal curvatures.  Its mean curvature is
	$$H(M_\theta) =  \frac{4 \sin(8\theta)}{\cos(8\theta) - 1},$$
	so that the \emph{minimal} isoparametric hypersurface has $\theta = \frac{\pi}{8}$.
	
	\subsubsection{The minimal hypersurface $M \subset S^9$} From now on, we fix $\theta = \frac{\pi}{8}$, and let $M := M_{\frac{\pi}{8}}$ denote the minimal isoparametric hypersurface in the $M_\theta$ family.  The induced metric $g_M$ on $M$ has
	\begin{equation}\label{eq:so5met}
		\begin{aligned}
			a_1 & = \textstyle \frac{1}{2} + \frac{1}{4}\sqrt{2}  & a_3 & = \textstyle 1 - \frac{1}{2}\sqrt{2} \\
			a_2 & = \textstyle \frac{1}{2} - \frac{1}{4}\sqrt{2} & a_4 & = \textstyle 1 + \frac{1}{2}\sqrt{2}.
		\end{aligned}
	\end{equation}
	By definition, $M$ consists of those skew-symmetric matrices in $\mathfrak{so}(5) = \R^{10}$ that are $\SO(5)$-conjugate to $h(\frac{\pi}{8}) = \cos(\frac{\pi}{8}) Y_1 + \sin(\frac{\pi}{8})Y_2$.  Thus,
	$$M = \textstyle \left\{ X \in \mathfrak{so}(5) \colon X \text{ has eigenvalues } 0,\, \pm i \cos(\frac{\pi}{8}),\, \pm i \sin(\frac{\pi}{8}) \right\}\!.$$
	It can be checked by direct computation that the Cartan-M\"{u}nzner polynomial of $M$ is the quartic $F \colon \R^{10} \to \R$ given by
	$$F(X) = 2 \left| X \wedge X \right|^2 - |X|^4,$$
	where we are viewing elements of $\R^{10} \cong \Lambda^2(\R^5)$ as $2$-forms.  Consequently, $M = F^{-1}(0) \cap S^9$.
	
	\begin{rmk} Later, we will need a specific adapted frame.  At the point $\mathbf{x} = \cos(\frac{\pi}{8})Y_1 + \sin(\frac{\pi}{8})Y_2$, we may choose $\mathbf{e}_1 = \textstyle -\sin(\frac{\pi}{8})Y_1 + \cos(\frac{\pi}{8})Y_2$ and
		\begin{align} \label{eq:SpecificFrame}
			\mathbf{e}_2 & = Y_4 & \mathbf{e}_4 & = Y_6 & \mathbf{e}_6 & = -Y_8 & \mathbf{e}_8 & = Y_{10} \\
			\mathbf{e}_3 & = -Y_3 & \mathbf{e}_5 & = -Y_5  & \mathbf{e}_7 & = Y_7 & \mathbf{e}_9 & = -Y_9. \notag
		\end{align}
	\end{rmk}
	
	\subsection{Decomposition of $L^2(M)$}
	
	\subsubsection{$\SO(5)$ representation theory}
	
	The complex Lie algebra $\mathfrak{so}(5; \C)$ has Cartan subalgebra $\mathfrak{t}^{\C} = \mathrm{span}_{\C}\{Y_1, Y_2\}$.  Let $\{\varepsilon_1, \varepsilon_2\}$ be the basis of $(\mathfrak{t}^{\C})^*$ that is dual to $\{iY_1, iY_2\}$.  It is well-known that $\mathfrak{so}(5; \C)$ has roots
	$$\{\pm \varepsilon_1, \pm \varepsilon_2, \pm(\varepsilon_1 + \varepsilon_2), \pm (\varepsilon_1 - \varepsilon_2)\}$$
	and fundamental weights $\omega_1 = \varepsilon_1$ and $\omega_2 = \frac{1}{2}(\varepsilon_1 + \varepsilon_2)$. \\
	\indent We let $V_{k,\ell}$ denote the irreducible complex $\SO(5)$-representation of highest weight $k\omega_1 + \ell \omega_2$, for $k,\ell \in \Z_{\geq 0}$.  The following examples will be particularly important in the sequel:
	\begin{align*}
		V_{1,0} & = \C^5 & V_{0,2} & = \mathfrak{so}(5;\C) = \Lambda^2(\C^5) \simeq \C^{10} \\
		V_{2,0} & = \Sym^2_0(\C^5) \simeq \C^{14} & V_{1,2} & \simeq \C^{35} \ \subset  \ \Lambda^2(\mathfrak{so}(5;\C)) \simeq \C^{45}
	\end{align*}
	Explicitly, let $\{v_1, \ldots, v_5\}$ be the standard basis of $V_{1,0} = \C^5$, let $\{Y_1, \ldots, Y_{10}\}$ be the basis of $V_{0,2} = \mathfrak{so}(5;\C)$ defined above, and let $\{ z_i z_j \colon 1 \leq i \leq j \leq 5 \}$ be the standard basis of $\Sym^2(\C^5)$.  A computation shows that the following are highest weight vectors in their respective representations:
	\begin{align*}
		v_2 + iv_3 & \in V_{1,0} & Y_9 + iY_{10} & \in V_{0,2} \\
		z_2^2  - z_3^2 + 2iz_2z_3 & \in V_{2,0} & (Y_3+iY_4) \wedge (Y_9 + iY_{10}) & \in V_{1,2}
	\end{align*}
	Accordingly, we define projection maps $\pi_{k,\ell} \colon V_{k, \ell} \to \C$ as follows:
	\begin{equation} \label{eq:HighestWeightProj-SO(5)}
		\begin{aligned}
			\pi_{1,0}( \textstyle \sum c_i v_i) & = c_2 + i c_3 \\
			\pi_{0,2}( \textstyle \sum c_i Y_i ) & = c_9 + ic_{10} \\
			\pi_{2,0}( \textstyle \sum c_{ij} z_i z_j) & = c_{22} - c_{33} + 2i c_{23}  \\
			\textstyle \pi_{1,2}(\sum (c_{i,j} - c_{j,i}) Y_i \wedge Y_j) & = (c_{9,3} - c_{3,9} + c_{4,10} - c_{10,4}) + i(c_{10,3} - c_{3,10} + c_{9,4} - c_{4,9}).
		\end{aligned}
	\end{equation}

	\subsubsection{The multiplicity formula (Step 1)}
	
	By Frobenius Reciprocity (Proposition \ref{prop:FrobRecGen}), there is an $\SO(5)$-invariant decomposition
	$$L^2(M; \C) = \bigoplus_{k, \ell \geq 0} V_{k, \ell}^{\oplus m(k, \ell)}$$
	where the multiplicity of $V_{k,\ell}$ is
	$$m(k, \ell) = \dim_{\C}\{ v \in V_{k,\ell} \colon h \cdot v = v, \ \forall h \in T^2\}.$$
	Since the isotropy group is the maximal torus $T^2$, the vector space $\{ v \in V_{k,\ell} \colon h \cdot v = v, \ \forall h \in T^2\}$ is simply the weight space of $V_{k,\ell}$ for weight $\mu = 0$, which is known \cite{bourbaki02} to have dimension
	\begin{align*}
		m(k,2\ell) & = \ell(k+1) + \left\lfloor \frac{k}{2} \right\rfloor + 1, \\
		m(k,2\ell+1) & = 0. 
	\end{align*}
	Explicitly, the decomposition reads as follows:
	\begin{align} \label{eq:L2Decomp}
		L^2(M; \C) & = V_{0,0} \oplus V_{1,0} \oplus  (V_{2,0} \oplus V_{2,0})  \oplus (V_{3,0} \oplus V_{3,0}) \oplus \cdots \\
		& \oplus (V_{0,2} \oplus V_{0,2}) \oplus (V_{1,2} \oplus V_{1,2} \oplus V_{1,2}) \oplus \cdots \notag \\
		& \oplus \cdots \notag
	\end{align}
	Our objective is to express $\Delta_M$ on each isotypical summand $V_{k,2\ell}^{\oplus m(k,\ell)} = V_{k,2\ell} \oplus \cdots \oplus V_{k,2\ell}$.
	
	\subsection{Explicit equivariant functions (Step 2a)}
	
	We now seek a sufficiently large supply of $\SO(5)$-equivariant functions $M \to V_{k, 2\ell}$, guided by the expansion (\ref{eq:L2Decomp}).  Two such functions are:
	\begin{enumerate}
		\item The complexified position function $\mathbf{x} \colon M \to V_{0,2} = \C^{10}$.
		\item The complexified normal vector $\mathbf{n} = \mathbf{e}_1 \colon M \to V_{0,2} = \C^{10}$.
	\end{enumerate}
	
	\noindent A third equivariant function is given by:
	
	\begin{prop} The image of $\mathbf{x} \wedge \mathbf{n} \colon M \to \Lambda^2(\C^{10})$ lies in $V_{1,2} \simeq \C^{35}$.
	\end{prop}
	
	\begin{proof} Note that $\mathbf{n} = \frac{1}{4}\nabla F$, where $\nabla$ is the gradient in $\R^{10}$.  Upon examination of
		$$\mathbf{x} \wedge \nabla F(\mathbf{x}) = (x_i Y_i) \wedge \left( \frac{\partial F}{\partial x_j} Y_j \right) = \left( x_i \frac{\partial F}{\partial x_j} - x_j \frac{\partial F}{\partial x_i} \right) Y_i \wedge Y_j,$$
		we see that it lies in a codimension $10$ subspace of $\Lambda^2(\C^{10})$.  In view of the $\SO(5)$-irreducible decomposition $\Lambda^2(\C^{10}) \cong V_{1,2} \oplus V_{0,2}$, we deduce the result.
	\end{proof}

	\indent Next, consider the quadratic function
	\begin{align*}
		q \colon \mathfrak{so}(5; \C) \cong \Lambda^2(\C^5) & \to \C^5 \\
		q(X) & := \ast(X \wedge X).
	\end{align*}
	Note that $|q(X)|^2 = 0$ if and only if $X$ is a decomposable $2$-form.  Explicitly, for an element $x_i Y_i \in \mathfrak{so}(5;\C) \cong \Lambda^2(\C^5)$, we have
	\begin{equation} \label{eq:Q-explicit}
		q(x_i Y_i) = \begin{bmatrix}
			2\,x_{{1}}x_{{2}}-{x_{{7}}}^{2}-{x_{{8}}}^{2}+{x_{{9}}}^{2}+{x_{{10}}}^{2} \\
			\sqrt {2}x_{{5}}x_{{7}}-\sqrt {2}x_{{5}}x_{{9}}+\sqrt {2} x_{{6}}x_{{8}}- \sqrt {2}x_{{6}}x_{{10}}-2\,x_{{2}}x_{{4}} \\
			-\sqrt {2}x_{{5}}x_{{8}}-\sqrt {2}x_{{5}}x_{{10}}+\sqrt {2}x_{{6}}x_{{7}}+\sqrt {2}x_{{6}}x_{{9}}+2\,x_{{2}}x_{{3}} \\
			-\sqrt {2}x_{{3}}x_{{7}}+\sqrt {2}x_{{3}}x_{{9}}+\sqrt {2}x_{{4}}x_{{8}}+\sqrt {2}x_{{4}}x_{{10}}-2\,x_{{6}}x_{{1}} \\
			-\sqrt {2}x_{{3}}x_{{8}}+\sqrt {2}x_{{3}}x_{{10}}-\sqrt {2}x_{{4}}x_{{7}}-\sqrt {2}x_{{4}}x_{{9}}+2\,x_{{5}}x_{{1}}
		\end{bmatrix}\!.
	\end{equation}
	
	Finally, define an quadratic map
	\begin{align*}
		s \colon \mathfrak{so}(5; \C) & \to \Sym^2_0(\C^5) \\
		s(X) & = \textstyle X^2 - \frac{1}{5}\mathrm{tr}(X^2)\,\mathrm{Id}_5
	\end{align*} 
	where $\mathrm{Id}_5$ is the $5 \times 5$ identity matrix. \\
	
	\noindent \textbf{Summary:} We now have five natural $\SO(5)$-equivariant functions on $M$:
	\begin{enumerate}
		\item The complexified position function $\mathbf{x} \colon M \to V_{0,2}$.
		\item The complexified normal vector $\mathbf{n} \colon M \to V_{0,2}$.
		\item The wedge product $\mathbf{x} \wedge \mathbf{n} \colon M \to V_{1,2}$.
		\item The quadratic function $q(\mathbf{x}) \colon M \to V_{1,0}$.
		\item The quadratic function $s(\mathbf{x}) \colon M \to V_{2,0}$.
	\end{enumerate}
	
	\begin{rmk} It turns out that the first four of these yield eigenfunctions because
		\begin{align*}
			\Delta \mathbf{x} & = -8\,\mathbf{x} & \Delta (\mathbf{x} \wedge \mathbf{n}) & = -32\,(\mathbf{x} \wedge \mathbf{n}) \\
			\Delta \mathbf{n} & = -24\,\mathbf{n} & \Delta q(\mathbf{x}) & = -16\,q(\mathbf{x}).
		\end{align*}
		This was observed in \cite{solomon90i,solomon90ii}, and can also be verified using the structure equations (\ref{eq:StrEqns}).  By contrast, using (\ref{eq:StrEqns}), one can compute
		$$\Delta \!\begin{bmatrix} s(\mathbf{x}) \\ s(\mathbf{n}) \end{bmatrix} = \begin{pmatrix} -18 & -2 \\ -14 & -46 \end{pmatrix} \begin{bmatrix} s(\mathbf{x}) \\ s(\mathbf{n}) \end{bmatrix}$$
		Upon diagonalizing this matrix, one finds that the components of $\left(-7 \pm 2\sqrt{14}\right) s(\mathbf{x}) + 7s(\mathbf{n})$ are eigenfunctions with eigenvalue $-32 \mp 4\sqrt{14}$.
	\end{rmk}
	
	\subsection{Decomposition of isotypical summands (Step 2b)}
	
	Recalling the maps $\pi_{k,\ell} \colon V_{k,\ell} \to \C$ in (\ref{eq:HighestWeightProj-SO(5)}), we define the following five functions $M \to \C$:
	\begin{align*}
		\kappa & := \pi_{1,0} \circ q(\mathbf{x}) &  \zeta & := \pi_{0,2} \circ \mathbf{x} & \rho & := \pi_{1,2} \circ (\mathbf{x} \wedge \mathbf{n}) \\
		\sigma & := \pi_{2,0} \circ s(\mathbf{x}) & \nu & := \pi_{0,2} \circ \mathbf{n}.
	\end{align*}
	Explicitly, from the expressions in (\ref{eq:HighestWeightProj-SO(5)}), we have:
	\begin{equation}\label{eq:hwvectsso5}
		\begin{aligned}
			\zeta(x_iY_i) & = x_9 + ix_{10}, \\
			\kappa(x_iY_i) & = \sqrt {2}x_{{5}}x_{{7}} - \sqrt {2}x_{{5}}x_{{9}} + \sqrt{2} x_{{6}}x_{{8}} - \sqrt{2} x_{{6}} x_{{10}} - 2 x_{{2}}x_{{4}} \\
			& \ \ \ \  + i \!\left( -\sqrt {2}x_{{5}}x_{{8}} - \sqrt {2}x_{{5}}x_{{10}}+ \sqrt{2} x_{{6}}x_{{7}} + \sqrt {2} x_{{6}} x_{{9}} + 2 x_{{2}}x_{{3}} \right)\!, \\
			\sigma(x_iY_i) & = -x_3^2 + x_4^2 - 2x_7 x_9 - 2x_8 x_{10} + i\!\left( -2x_3 x_4 - 2x_7 x_{10} + 2 x_8 x_9 \right)\!, \\
			\nu(x_iY_i) & = \frac{1}{4}\left( \frac{\partial F}{\partial x_9} + i \frac{\partial F}{\partial x_{10}} \right)\!. \\
			\rho(x_iY_i) & = \frac{1}{4}(x_9 + ix_{10}) \left( \frac{\partial F}{\partial x_3} + i\frac{\partial F}{\partial x_4} \right) -  \frac{1}{4}\left( \frac{\partial F}{\partial x_9} + i \frac{\partial F}{\partial x_{10}} \right) (x_3 + ix_4).
		\end{aligned}
	\end{equation}
	Now, for $u,v,x,y,z \in \Z_{\geq 0}$, we define
	$$U_{u,v,x,y,z} := \mathrm{span}_{\C}\! \left\{ g^\#( \zeta^u \kappa^v \nu^x \rho^y \sigma^z) \colon g \in \SO(5) \right\} \subset L^2(M;\C),$$
	where $g^\#$ denotes the $\SO(5)$-action on $L^2(M; \C)$ discussed in $\S$\ref{subsec:Method}.  By Proposition \ref{prop:SpanHighestWeight}(c), we have an isomorphism of  $\SO(5)$-modules
	$$U_{u,v,x,y,z} \cong V_{v+y+2z,\,2(u+x+y)}.$$
	\indent We will now prove that, for example,
	\begin{align*}
		V_{2,0}^{\oplus m(2,0)} & \cong U_{0,2,0,0,0} + U_{0,0,0,0,1} = \mathrm{span}_{\C}\{g^\sharp (\kappa^2), \, g^\sharp (\sigma) \colon g \in \SO(5)\} \\
		V_{1,2}^{\oplus m(1,2)} & \cong U_{0,0,0,1,0} + U_{1,1,0,0,0} + U_{0,1,1,0,0} = \mathrm{span}_{\C}\{g^\sharp (\rho),\, g^\sharp (\zeta \kappa),\, g^\sharp (\kappa \nu) \colon g \in \SO(5)\}
	\end{align*}
	More generally:
	
	\begin{prop} \label{prop:Isotypical-SO5} Define
		$$W(k,\ell) := \mathrm{span}\left\{ U_{u,v,x,y,z} \colon u,v,x,y,z \in \Z_{\geq 0},\ y \in \{0,1\},\ v+y+2z = k,\ u+x+y = \ell \right\}\!.$$
		Then
		$$V_{k, 2\ell}^{\oplus m(k, 2\ell)} \cong W(k,\ell).$$
		Consequently,
		$$L^2(M;\C) = \bigoplus_{k, \ell \geq 0} W(k, \ell).$$
	\end{prop}
	
	\noindent Proposition \ref{prop:Isotypical-SO5} is a consequence of the following two lemmas.
	
	\begin{lem} \label{lem:Trigonometry} Fix $k$ and let 
		\begin{align*}
			f_i(t) & = (\sin t)^{k-2i} (\cos t)^{2i}, & g_i(t) & = (\sin t)^{k-1-2i} (\cos t)^{2i+1}.
		\end{align*}
		Then the set of functions $\{ f_0, \ldots, f_{\lfloor k/2 \rfloor} \} \cup \{ g_0, \ldots, g_{\lfloor (k-1)/2 \rfloor} \}$ is linearly independent.
	\end{lem}
	
	\begin{proof} If $k$ is even, each $f_i$ is even and each $g_i$ is odd, and if $k$ is odd, each $f_i$ is odd and each $g_i$ is even, so it suffices to prove that the sets $\{f_i(t)\}$ and $\{g_i(t)\}$ are individually linearly independent.  For this, note that
		\begin{align*}
			f_i+f_{i+1} &= (\sin t)^{k-2i} (\cos t)^{2i} + (\sin t)^{k-2i-2} (\cos t)^{2i+2} = (\sin t)^{k-2i-2} (\cos t)^{2i}, \\
			g_i+g_{i+1} &= (\sin t)^{k-1-2i} (\cos t)^{2i+1} + (\sin t)^{k-3-2i} (\cos t)^{2i+3} = (\sin t)^{k-2i-3} (\cos t)^{2i+1},
		\end{align*}
		so that
		\begin{align*}
			\sum_{i=0}^p \binom{p}{i} f_i(t) & = (\sin t)^{k-2p} & \sum_{i=0}^p \binom{p}{i} g_i(t) & = (\sin t)^{k-1-2p} (\cos t).
		\end{align*}
		Since the sets $\{ (\sin t)^{k-2p} \colon 0 \leq p \leq \lfloor k/2 \rfloor\}$ and $\{ (\sin t)^{k-1-2p} (\cos t) \colon 0 \leq p \leq \lfloor (k-1)/2 \rfloor\}$ are linearly independent, it follows that $\{f_i(t)\}$ and $\{g_i(t)\}$ are, too.
	\end{proof}

 Define the following subsets of $W(k, \ell)$:
		\begin{align*}
			\mathcal{U}_{k,\ell} & := \left\{ \kappa^{k -2 i} \sigma^{i} \zeta^{\ell-j} \nu^j   \colon 0 \leq i \leq \lfloor k/2 \rfloor, \, 0 \leq j \leq \ell \right\}, \\
			\mathcal{V}_{k,\ell} & := \left\{  \rho \kappa^{k-1 -2 i} \sigma^{i} \zeta^{\ell-1-j} \nu^j  \colon 0 \leq i \leq \lfloor (k-1)/2 \rfloor, \, 0 \leq j \leq \ell - 1 \right\}.
		\end{align*}
	
	\begin{lem} \label{lem:LinIndep-SO5} The set $\mathcal{U}_{k,\ell} \cup \mathcal{V}_{k,\ell}$ is linearly independent.  Moreover,
		$$\dim( \mathrm{span}\{ \mathcal{U}_{k,\ell} \cup \mathcal{V}_{k,\ell}\}) = m(k,2\ell).$$
	\end{lem}
	
	\begin{proof} We will argue by induction, reasoning as follows. Let $T(k,l)$ be the statement
		$$\text{``$\mathcal{U}_{k,l} \cup \mathcal{V}_{k,l}$ is linearly independent."}$$
		The base case $T(0,0)$ is trivial (note that $\mathcal{V}_{0,0}$ is empty). We will first prove $T(k,\ell-1) \implies T(k,\ell)$ for $\ell >0$.  Then we will prove both $T(2r,0) \implies T(2r+1,0)$ and $T(2r,0) \implies T(2r+2,0)$ for $r>0$.
		
		Before beginning, note that the 1-parameter family of points
		\begin{equation*}
			p_t = 2^{-1/4} \left( \cos t \, Y_3 - \sin t \, Y_5 \right) + \tfrac{1}{2} (4 - 2 \sqrt{2})^{1/2} Y_7 \in \mathfrak{so}(5) \simeq \R^{10}
		\end{equation*}
		satisfy $F(p_t) = 0$ and hence are all elements of $M.$ Moreover,
		\begin{align*}
			\zeta ( p_t ) &= 0, & \nu(p_t) &= -2^{-1/2} (4 - 2 \sqrt{2})^{1/2}, \\
			\kappa (p_t) & = - 2^{-3/4} (4 - 2 \sqrt{2})^{1/2} \sin t, &  \sigma(p_t) &= - \sqrt{2} \cos^2 (t) \\
			\rho (p_t) &= 2^{-3/4} (4 - 2 \sqrt{2})^{1/2} \cos t.
		\end{align*}
		
		First, fix $\ell > 0$ and suppose that we have a linear relation in $\mathcal{U}_{k,\ell} \cup \mathcal{V}_{k,\ell}$ written as
		\begin{equation} \label{eq:supposedlinrel}
			\sum_i \sum_j a_{ij} \kappa^{k -2 i} \sigma^{i} \zeta^{\ell-j} \nu^j +  \sum_i \sum_j b_{ij} \rho \kappa^{k-1-2 i} \sigma^{i} \zeta^{\ell-1-j} \nu^j = 0.
		\end{equation}
		Evaluating at $p_t$ and absorbing constants into $a_{ij}, b_{ij}$, we get
		\begin{equation*}
			\sum_i a'_{i\ell} \left( \sin t \right)^{k-2i} \left( \cos t \right)^{2i} +  \sum_i b'_{i, \ell-1} \left( \sin t \right)^{k-1-2i} \left( \cos t \right)^{2i+1} = 0,
		\end{equation*}
		where the $a'_{i \ell}$ and $b'_{i, \ell-1}$ are obtained from $a_{i \ell}$ and $b_{i, \ell-1}$ by rescaling by some non-zero constant (which will depend on $i$, $k$, $\ell$). Now, by Lemma \ref{lem:Trigonometry}, the set of functions
		\begin{equation*}
			\lbrace \left( \sin t \right)^{k-2i} \left( \cos t \right)^{2i} \mid 0 \leq i \leq \lfloor k/2 \rfloor \rbrace \cup \lbrace \left( \sin t \right)^{k-1-2i} \left( \cos t \right)^{2i+1} \mid 0 \leq i \leq \lfloor (k-1)/2 \rfloor \rbrace
		\end{equation*}
		is linearly independent. Therefore, $a'_{i\ell}$ and $b'_{i, \ell-1}$ must vanish for all $i$, which implies $a_{i \ell}$ and $b_{i,\ell-1}$ vanish for all $i$, too. Consequently, the left-hand side of (\ref{eq:supposedlinrel}) is of the form $\zeta F$, where $F \in \mathrm{span}\{\mathcal{U}_{k, \ell-1} \cup \mathcal{V}_{k, \ell-1}\}$.  Since $\zeta$ is not the zero function, we obtain $F = 0$. Therefore, linear independence of $\mathcal{U}_{k, \ell-1} \cup \mathcal{V}_{k, \ell-1}$ implies linear independence of $\mathcal{U}_{k, \ell} \cup \mathcal{V}_{k, \ell}$. \\
		\indent Next, we show that $\mathcal{U}_{k,0} \cup \mathcal{V}_{k,0}$ is linearly independent. In fact, $\mathcal{V}_{k,0} = \varnothing,$ so we need to show $\mathcal{U}_{k,0}$ is linearly independent.  In the case where $k = 2r+1$ is odd, any linear relation in $\mathcal{U}_{2r+1, 0}$ is of the form $\kappa F = 0$, where $F \in \mathrm{span}\{ \mathcal{U}_{2r, 0}\}$.  Thus, linear independence of $\mathcal{U}_{2r,0}$ implies that of $\mathcal{U}_{2r+1,0}$. In the case where $k = 2r$ is even, consider a linear relation in $\mathcal{U}_{2r,0}$, which we may write as
		\begin{equation*}
			\sum_{i=0}^r a_i \kappa^{2r-2i} \sigma^i = 0.  
		\end{equation*}
		Evaluating at the point $p_0$, we have $a_r = 0,$ so the linear relation is divisible by $\kappa^2$.  Thus, linear independence of $\mathcal{U}_{2r-2,0}$ implies linear independence of $\mathcal{U}_{2r,0}$.  \\
		\indent Finally, we establish the last claim.  For this, we note that the number of elements in $\mathcal{U}_{k,\ell} \cup \mathcal{V}_{k,\ell}$ is
		\begin{align*}
			\textstyle \left(\left\lfloor \frac{k}{2} \right\rfloor + 1\right)(\ell + 1) + \left( \left\lfloor \frac{k-1}{2} \right\rfloor + 1 \right)\ell & \textstyle = \ell \left( \left\lfloor \frac{k}{2} \right\rfloor + \left\lfloor \frac{k-1}{2} \right\rfloor + 2 \right) + \left\lfloor \frac{k}{2} \right\rfloor + 1 \\
			& \textstyle = \ell(k+1) + \left\lfloor \frac{k}{2} \right\rfloor + 1 \\
			& = m(k, 2\ell).
		\end{align*}
	\end{proof}
	
	\indent Lemma \ref{lem:LinIndep-SO5} implies that
	$$W(k,\ell) = \mathrm{span}\{ g^\sharp( \mathcal{U}_{k,\ell} \cup \mathcal{V}_{k,\ell}) \colon g \in \SO(5) \} \cong V_{k,2\ell}^{\oplus m(k,2\ell)},$$
	thereby establishing Proposition \ref{prop:Isotypical-SO5}.
	
	\subsection{Computing $\Delta_M$ on isotypical summands (Step 3)}
	
	For fixed $k, \ell \geq 0$, define functions
	\begin{align*}
		F_i & = \kappa^{k-2i} \sigma^i, \ \ \ 0 \leq i \leq \lfloor k/2 \rfloor, \\
		G_j & = \zeta^{\ell - j} \nu^j, \ \ \ \ 0 \leq j \leq \ell.
	\end{align*}
 Note that, although it is not made clear in the notation, $F_i$ and $G_j$ depend on $k$ and $\ell$ respectively.
By definition of $\mathcal{U}_{k,\ell}$ and $\mathcal{V}_{k+1,\ell+1},$
	\begin{align*}
		& \{F_i G_j \colon 0 \leq i \leq \lfloor k/2 \rfloor, 0 \leq j \leq \ell\} \text{ is a basis of } \mathrm{span}(\mathcal{U}_{k,\ell}), \text{ and} \\
		& \{\rho F_i G_j \colon 0 \leq i \leq \lfloor (k-1)/2 \rfloor, 0 \leq j \leq \ell - 1\} \text{ is a basis of } \mathrm{span}(\mathcal{V}_{k+1,\ell+1}).
	\end{align*}
	The purpose of this section is to compute the Laplacians $\Delta (F_i G_j)$ and $\Delta (\rho F_i G_j)$.  Recalling our convention that the Laplacian is negative semi-definite, the product rule gives
	\begin{align}
		\Delta (F_i G_j) & = F_i \Delta G_j + 2 \nabla F_i \cdot \nabla G_j + G_j \Delta F_i \label{eq:LaplaceFG} \\
		\Delta (\rho F_i G_j) & = \rho \Delta (F_i G_j) + 2 \nabla \rho \cdot \nabla(F_i G_j) + F_i G_j \Delta \rho. \label{eq:LaplacePFG}
	\end{align}
	Thus, repeatedly applying the product rule, we see that we require the pairwise dot products of the gradients $\nabla \zeta$, $\nabla \kappa$, etc., as well as the Laplacians $\Delta \zeta$, $\Delta \kappa$, etc.  These are given by:
	
	\begin{prop} \label{prop:DotProdGradSO(5)} We have
		\begin{align*}
			\nabla \zeta \cdot \nabla \zeta & = - \zeta^2 - \nu^2 & \nabla \kappa \cdot \nabla \kappa &= - 8 \kappa^2 + 4 \sigma \\
			\nabla \zeta \cdot \nabla \nu &= 2 \zeta^2 - 6 \zeta \nu & \nabla \kappa \cdot \nabla \sigma &= -4 \kappa^3 \\
			\nabla \nu \cdot \nabla \nu &= 3 \zeta^2 - 4 \zeta \nu - 9 \nu^2 & \nabla \sigma \cdot \nabla \sigma &= -4 \kappa^4 +8 \kappa^2 \sigma -8 \sigma^2
		\end{align*}
		and
		\begin{align*}
			\nabla \zeta \cdot \nabla \rho &= - 4 \zeta \rho, & \nabla \zeta \cdot \nabla \kappa &= -2\zeta\kappa-2 \nu \kappa  \\
			\nabla \nu \cdot \nabla \rho &= -12 \nu \rho & \nabla \zeta \cdot \nabla \sigma &= -2 \zeta \sigma - 2 \nu \kappa^2 + 2 \nu \sigma  \\
			\nabla \kappa \cdot \nabla \rho &= -8\kappa \rho & \nabla \nu \cdot \nabla \kappa &= -2\zeta \kappa - 6 \nu \kappa  \\
			\nabla \sigma \cdot \nabla \rho &= - 8 \sigma \rho & \nabla \nu \cdot \nabla \sigma &= -6\zeta \kappa^2 + 2 \zeta \sigma - 6 \nu \sigma
		\end{align*}
		and
		\begin{align*}
			\Delta \zeta & = -8\zeta  &    \Delta \kappa &= -16 \kappa & \Delta \rho & = -32 \rho \\
			\Delta \nu & = -24\nu &   \Delta \sigma &= -4 \kappa^2 - 16 \sigma.
		\end{align*}
	\end{prop}
	
	\begin{proof} Suppose $\mu_1, \mu_2 \in \left\lbrace \zeta, \nu, \rho, \kappa, \sigma \right\rbrace$ have weights $w_1, w_2$ respectively. We know by Proposition \ref{prop:SpanHighestWeight}(d) that $\nabla \mu_1 \cdot \nabla \mu_2$ is a highest weight function of weight $w_1 + w_2$. We have a basis for the space of highest weight functions for a given weight, so $\nabla \mu_1 \cdot \nabla \mu_2$ can be expressed uniquely as a linear combination of elements of this basis. On the other hand, using (\ref{eq:StrEqns}), the explicit formulas of (\ref{eq:hwvectsso5}), and the technique described in $\S$\ref{subsubsec:Differentiation}, we may obtain an expression for $\nabla \mu_1 \cdot \nabla \mu_2$ in terms of the components of an adapted frame $\left( \mathbf{x}, \mathbf{e}_1, \ldots, \mathbf{e}_9 \right).$ By evaluating this expression on a sufficiently general circle of points in $M,$ we may determine the coefficients of the basis expansion of $\nabla \mu_1 \cdot \nabla \mu_2.$ The calculation of the Laplacians works similarly.
		
		We illustrate our technique on $\nabla \zeta \cdot \nabla \kappa.$ Let $\mathbf{f} \in \psi_{\pi/8}\left(\SO(5)\right)$ be the frame given in (\ref{eq:SpecificFrame}), and let $A \in \SO(5)$ be the matrix 
		\begin{equation*}
			A = \left[ \begin {array}{ccccc} \frac{1}{2}\,\sqrt {2}&0&0&0&\frac{1}{2}\,\sqrt {2}
			\\ \noalign{\medskip}0&\frac{1}{2}&0&-\frac{1}{2}\,\sqrt {3}&0\\ \noalign{\medskip}0&0
			&1&0&0\\ \noalign{\medskip}0&\frac{1}{2}\,\sqrt {3}&0&\frac{1}{2}&0
			\\ \noalign{\medskip}-\frac{1}{2}\,\sqrt {2}&0&0&0&\frac{1}{2}\,\sqrt{2}\end {array}
			\right].
		\end{equation*}
		Then the set of elements $\Ad(\exp(tY_3) A) \mathbf{f} \in \SO(10), $ $t \in \left[0, 2\pi \right],$ gives a circle of points in $M$ together with an adapted frame at each point. On this circle we have
		\begin{equation}\label{eq:SOhwfeval}
			\begin{aligned}
				\zeta &= \tfrac{1}{4} \left( -i\sqrt {3}\cos\!\left( t \right) +\sin \!\left( t \right)  \right) \sin\! \left( \tfrac{1}{8}\pi \right) - \tfrac{i \sqrt{6}}{4}\cos \!\left( \tfrac{1}{8}\pi \right), \\
				\nu &= \tfrac{1}{4} \left( -i\sqrt {3}\cos\! \left( t \right) +\sin \!\left( t \right)  \right) \sin\!\left( \tfrac{3}{8}\pi \right) + \tfrac{i \sqrt{6}}{4}\sin \!\left( \tfrac{1}{8}\pi \right),  \\
				\kappa &= -\tfrac{1}{2} \sin\! \left( t \right), \\
				\rho &= - \tfrac{1}{8}\sqrt {3}\,\sin\! \left( t \right) \cos\! \left( t \right) + \tfrac{1}{8}i \cos^2 \!\left( t \right) - \tfrac{1}{2}i, \\
				\sigma &= -\tfrac{1}{4} \cos^2\!\left( t \right) + \tfrac{3}{8}\,\sqrt {2}
				+ \tfrac{1}{4},
			\end{aligned}
		\end{equation}
		and we may compute
		\begin{equation}\label{eq:SOdoteval}
			\nabla \zeta \cdot \nabla \kappa = -\tfrac{i\sqrt{6}}{4} \sin \left( t \right) \cos \left( \tfrac{1}{8}\pi \right) 
			\cos \left( t \right)- \tfrac{i \sqrt{3}}{2} \sin \left( \tfrac{1}{8}\pi \right) 
			\sin \left( t \right) + \tfrac{\sqrt{2}}{4} \cos \left( \tfrac{1}{8}\pi \right) \sin^2\! \left( t \right).
		\end{equation}
		On the other hand, the function $\nabla \zeta \cdot \nabla \kappa$ has weight $(1,2),$ so it must be a linear combination of $\zeta \kappa,$ $\nu \kappa,$ and $\rho.$ Comparing (\ref{eq:SOhwfeval}) and (\ref{eq:SOdoteval}), we find
		\begin{equation*}
			\nabla \zeta \cdot \nabla \kappa = -2\zeta\kappa-2 \nu \kappa.
		\end{equation*}
		The other formulas are calculated in a similar fashion.
	\end{proof}
	
	\begin{prop} \label{prop:LaplaceProductsSO5} We have
		\begin{equation*}
			\begin{aligned}
				\Delta ( F_i G_j ) & = -4 i (i-1) F_{i-2}G_{j} -12 ij F_{i-1}G_{j-1} - 4 i (2k+3 - 6 i) F_{i-1}G_j - 4 i (\ell-j) F_{i-1} G_{j+1} \\
				& + 3 j (j-1) F_i G_{j-2} + 4 j ( \ell -k + 3i - 2j +1 ) F_i G_{j-1} - 4 (\ell - j )(k-3i) F_i G_{j+1} \\
				& - (\ell-j)(\ell-j-1) F_i G_{j+2} + 4 (k-2i)(k - 2 i - 1) F_{i+1} G_j \\
				& + ( -40 i^2 + 8 ij + 32 ik + 4 i\ell +  2 j^2 - 8kj - 10 \ell j - 8 k^2 - 4 k\ell - \ell^2 + 8 i - 8 j - 8 k - 7 \ell ) F_i G_j
			\end{aligned}
		\end{equation*}
		and
		\begin{equation*}
			\begin{aligned}
				\Delta ( \rho F_i G_j ) & = \rho \Delta(F_iG_j) - (32 + 16k - 16 i +8 \ell + 16 j) \rho F_i G_j. \\
			\end{aligned}
		\end{equation*}
	\end{prop}
	
		\begin{proof} Using Proposition \ref{prop:DotProdGradSO(5)}, we calculate
			\begin{align}
				\label{eq:DeltaF}  \Delta F_i & = -4i(i-1)F_{i-2} - 4i\left( 2k + 3 -6i \right)F_{i-1} \\
				\notag & \ \ \ \ +  \left( -40i^2 + 32ik + 8i -8k^2 - 8k \right)F_i +  4(k-2i)(k-2i-1) F_{i+1}\\
				\label{eq:DeltaG} \Delta G_j & = 3j(j-1)G_{j-2} + 4j\left(\ell -2j + 1 \right) G_{j-1}  \\
				\notag & + \left( 2j^2 - 10j\ell - 8j - \ell^2 - 7\ell \right) G_j - (\ell-j)(\ell-j-1) G_{j+2} 
			\end{align}
			Next, we compute
			\begin{align*}
				\nabla F_i & = i \kappa^{k-2i} \sigma^{i-1} \nabla \sigma + (k-2i) \kappa^{k-2i-1} \sigma^i \nabla \kappa \\
				\nabla G_j & = j \zeta^{\ell-j} \nu^{j-1} \nabla \nu + (\ell-j)\zeta^{\ell-j-1} \nu^j \nabla \zeta
			\end{align*}
			Taking their dot product, and applying Proposition \ref{prop:DotProdGradSO(5)}, we find
			\begin{align} \label{eq:NablaFG}
				2\nabla F_i \cdot \nabla G_j & = -4\left( -2ij - i\ell + 2jk + k\ell \right) F_i G_j  + 4j\left( 3i - k \right) F_i G_{j-1} \\
				\notag & \ \ \  + 4(\ell - j) \left( 3i - k \right) F_i G_{j+1} - 12ij F_{i-1} G_{j-1} - 4i (\ell-j) F_{i-1} G_{j+1}
			\end{align}
			Finally, plugging (\ref{eq:DeltaF}), (\ref{eq:DeltaG}), and (\ref{eq:NablaFG})  into (\ref{eq:LaplaceFG}) yields the first claim. \\
			\indent For the second claim, we calculate
			$$2 \nabla \rho \cdot \nabla(F_i G_j) = -\rho F_i G_j(-16i + 16j + 16k + 8\ell).$$
			Applying (\ref{eq:LaplacePFG}) gives the result.
		\end{proof}
		
		\subsection{The index and nullity (Step 4)} 
		
		\indent For $k, \ell \geq 0$, let $Y(k, 2\ell) \subset L^2(M;\R)$ denote the real subspace spanned by the real and imaginary parts of the functions in $\mathcal{U}_{k,\ell} \cup \mathcal{V}_{k,\ell}$, and let $\Delta_{k,2\ell} := \Delta|_{Y(k,2\ell)}$ denote the Laplacian of $g_M$ restricted to $Y(k,2\ell)$. \\
		
		\indent We are almost ready to compute the Morse index and nullity of $M \subset S^9$.  For this, we need an exhaustive list of the eigenvalues (and their multiplicities) of $-\Delta$ that are less than $32$.  Now, Proposition \ref{prop:LaplaceProductsSO5} yields the matrix of $\Delta_{k,2\ell}$ with respect to the basis $\mathrm{Re}(\mathcal{U}_{k,\ell} \cup \mathcal{V}_{k,\ell}) \cup \mathrm{Im}(\mathcal{U}_{k,\ell} \cup \mathcal{V}_{k,\ell})$. Thus, it remains to understand which $Y(k,2\ell)$ house eigenvalues less than 32. \\
		\indent For this, we will apply Theorem \ref{thm:MainBound} to the homogeneous space $N = \SO(5)/\mathrm{T}^2$.  To that end, we let $g_1 = \varphi_{\pi/8}^*g_M$ denote the pullback of the minimal isoparametric metric to $N$, and note that $\Delta_{k,2\ell}$ is a self-adjoint endomorphism with respect to the $L^2$ inner product on $Y(k,2\ell)$ given by $(f_1, f_2) = \int_N f_1 f_2\,\vol_{g_1}$.  Next, we let $g_2 = g_\mathrm{nor}$ denote the normal homogeneous metric, and let $L_{k, 2\ell}$ denote the Laplacian of $g_{\mathrm{nor}}$ restricted to $Y(k, 2\ell)$.  By standard theory \cite{yamaguchi79}, we have 
		$$L_{k, 2\ell} = -\mu(k, 2\ell)\,\mathrm{Id},$$
		where $\mu(k, 2\ell)$ is the eigenvalue of the Casimir of $\SO(5)$ on $V_{k,2\ell}$.  By Freudenthal's formula
		$$\mu(k, 2\ell) = (k+\ell)(k+\ell+3) + \ell(\ell+1).$$
		We now have:
		
		\begin{prop} \label{prop:BoundsSO5} Let $\lambda$ be an eigenvalue of $-\Delta_{k, 2\ell}$.  Then: 
			\begin{align*}
				\left( 2 - \sqrt{2} \right) \left[ (k+\ell)(k+\ell+3) + \ell(\ell+1) \right] & \leq \lambda \leq \left( 4 + 2 \sqrt{2} \right)  \left[ (k+\ell)(k+\ell+3) + \ell(\ell+1) \right]
			\end{align*}
		\end{prop}
		
		\begin{proof} As before, we let $\pi \colon \SO(5) \to N$ denote the coset projection.  Recall from (\ref{eq:so5met}) that the normal homogeneous metric $g_{\mathrm{nor}}$ and the minimal isoparametric metric $g_1$ pull back to $\SO(5)$ as
			\begin{align*}
				\pi^*g_{\mathrm{nor}} & = \omega_3^2 + \cdots + \omega_{10}^2 \\
				\pi^*g_1 & = \textstyle \frac{1}{4}(2 + \sqrt{2})(\omega_3^2 + \omega_4^2) + \frac{1}{4}(2 - \sqrt{2})(\omega_5^2 + \omega_6^2) \\
				& \ \ \ \ \textstyle + \frac{1}{2}(2 - \sqrt{2})(\omega_7^2 + \omega_8^2) + \frac{1}{2}(2 + \sqrt{2})(\omega_9^2 + \omega_{10}^2).
			\end{align*}
			In particular, the smallest and largest eigenvalues of $g_1$ with respect to $g_{\mathrm{nor}}$ are $\frac{1}{4}(2-\sqrt{2})$ and $\frac{1}{4}(4 + \sqrt{2})$.  Taking reciprocals, the smallest and largest eigenvalues of the dual metric $g_1^{-1}$ with respect to $g_{\mathrm{nor}}^{-1}$ are $r_{\min} = 2-\sqrt{2}$ and $r_{\max} = 4 + 2 \sqrt{2}$.  Applying Theorem \ref{thm:MainBound} yields the result.
		\end{proof}
		
		\indent We can now understand which eigenspaces of $-\Delta$ have eigenvalue at most $32$.  Indeed, by Proposition \ref{prop:BoundsSO5}, we need only consider the operators $-\Delta_{k,2\ell}$ for $(k,2\ell)$ satisfying
		$$\left( 2 - \sqrt{2} \right) \left[ (k+\ell)(k+\ell+3) + \ell(\ell+1) \right] \leq 32.$$
		That is, we need only consider the following values of $(k,2\ell)$:
		\begin{equation*}
			\begin{aligned}
				(0,0) & & (0,2) & & (0,4) & & (0,6) & & (0,8) \\
				(1,0) & & (1,2) & & (1,4) & & (1,6) & & \\
				(2,0) & & (2,2) & & (2,4) & & (2,6) & & \\
				(3,0) & & (3,2) & & (3,4) & &  & \\
				(4,0) & & (4,2) & & & & \\
				(5,0) & &  & & & & \\
				(6,0) & & & & & &
			\end{aligned}
		\end{equation*}
		Using the matrix representations of $-\Delta_{k,2\ell}$ provided by Proposition \ref{prop:LaplaceProductsSO5}, we find that the following are all the eigenvalues of $-\Delta$ less than or equal to $32$:
		$$\begin{tabular}{| l | c | c |} \hline
			Eigenvalue of $-\Delta$ & $(k,2\ell)$ & Multiplicity \\ \hline \hline
			$\lambda_1 = 0$ & $(0,0)$ & 1 \\ \hline
			$\lambda_2 = 8$ & $(0,2)$ & 10 \\ \hline
			$\lambda_3 = 16$ & $(1,0)$ & 5 \\ \hline 
			$\lambda_4 = 32 - 4 \sqrt{14} \approx 17.03$ & $(2,0)$ & 14 \\ \hline
			$\lambda_5 \approx 18.20$ & $(0,4)$ & 35 \\ \hline
			$\lambda_6 = 24$ & $(0,2)$ & 10 \\ \hline
			$\lambda_7 = 40 - 4 \sqrt{10} \approx 27.35$ & $(1,2)$ & 35 \\  \hline
			$\lambda_8 \approx 28.62$ & $(2,2)$ & 81 \\ \hline
			$\lambda_9 = 64 - 4 \sqrt{70} \approx 30.53$ & $(0,6)$ & 84 \\ \hline \hline
			$\lambda_{10} = 32$ & $(1,2)$ & 35 \\ \hline
		\end{tabular}$$
		\noindent Here, the eigenvalues $\lambda_5$ and $\lambda_8$ are computed as follows:
		\begin{align*}
			\lambda_5 & = \text{min. root of} \:\:\: {x}^{3}-128\,{x}^{2}+4896\,x-52736 \\
			\lambda_8 & = \text{min. root of} \:\:\: {x}^{3}-176\,{x}^{2}+9120\,x-140288
		\end{align*}
		Also, the multiplicity an eigenvalue is given by $\dim(V_{k,2\ell})$.  By Weyl's dimension formula, we have
		\begin{equation*}
			\dim V_{k,2\ell} = \tfrac{1}{6} \left( k+2\,\ell+2 \right)  \left( k+1 \right)  \left( 2\,k+2\,\ell+3 \right)  \left( 2\,\ell+1 \right)\!.
		\end{equation*}
		The information in the above table now yields:
		
		\begin{thm} The index of the $\SO(5)$-invariant minimal isoparametric hypersurface in $S^9$ is 275. Furthermore, the nullity is 35 and every Jacobi field is rotational.
		\end{thm}
		
		\begin{proof} Summing the dimensions of the eigenspaces of $-\Delta$ having eigenvalue less than $32$, we obtain
			\begin{align*}
				\text{Ind}(M) &  = 1 + 10 + 5 + 14 + 35 + 10 + 35 + 81 + 84 = 275.
			\end{align*}
			Moreover, since $\dim(\SO(10)) - \dim(\SO(5)) = 35$, the space of rotational Jacobi fields is $35$-dimensional.  Since the nullity of $M$ is also equal to $35$, we deduce that every Jacobi field is rotational. \\
			\indent (Alternatively, one can observe that the eigenspace of $-\Delta$ having eigenvalue 32 corresponds to $\rho = \pi_{1,2} \circ (\mathbf{x} \wedge \mathbf{n})$, i.e., is spanned by the components of $\mathbf{x} \wedge \mathbf{n}$.  This also shows that every Jacobi field on $M$ is rotational.)
		\end{proof}
		
		\begin{thm} The Laplace spectrum of the minimal isoparametric metric on $\SO(5)/\mathrm{T}^2$ contains eigenvalues that are not expressible in radicals.
		\end{thm}
		
		\begin{proof}
			For $k = 8, \ell = 0,$ the Laplacian $\Delta$ acts on $\operatorname{span}{\mathcal{U}_{8,0}}$ as
			\begin{equation*}
				\Delta_{8,0} = \left[ \begin {array}{ccccc} -576&224&0&0&0\\ \noalign{\medskip}-52&-
				352&120&0&0\\ \noalign{\medskip}-8&-56&-208&48&0\\ \noalign{\medskip}0
				&-24&-12&-144&8\\ \noalign{\medskip}0&0&-48&80&-160\end {array}
				\right].
			\end{equation*}
			The characteristic polynomial of this matrix is 
			\begin{equation*}
				{x}^{5}+1440\,{x}^{4}+782464\,{x}^{3}+200810496\,{x}^{2}+24403705856\,x+1126904627200,
			\end{equation*}
			which has Galois group $\mathrm{S}_5,$ so its roots cannot be expressed in radicals.
		\end{proof}
		
		\begin{rmk} 
			This behaviour is appears to be typical. Aside from small values of $k$ and $\ell,$ in many examples we found that the characteristic polynomial of the Laplacian $\Delta_{k,2\ell}$ has Galois group $\mathrm{S}_{m(k,2\ell)}.$
		\end{rmk}

		\section{The $(\SO(3) \times \SO(2))$-Invariant Minimal Hypersurface in $S^5$}
		
		\indent In this section, we implement the steps of $\S$\ref{subsec:Method} to compute the Laplace spectrum of the $(\SO(3) \times \SO(2))$-invariant minimal hypersurface in $S^5$, and thereby calculate its Morse index and nullity.  We continue with the notation of $\S$\ref{subsec:MovingFrames}, taking $n = 5$, $\mathrm{G} = \SO(3) \times \SO(2)$, and $\mathrm{H} = \Z_2$.  Moreover, $\eta$ and $\omega$ will denote the Maurer-Cartan forms on $\SO(6)$ and $\SO(3) \times \SO(2)$, respectively.
		
		\subsection{The minimal hypersurface $M \subset S^5$}
		
		\subsubsection{The $\SO(3) \times \SO(2)$-action on $S^5$}
		
		We will view the ambient space as $\R^6 = \Hom(\R^3, \R^2)$.  Letting $\{e_1, e_2, e_3\}$ be the standard basis of $\R^3$, and $\{E_1, E_2\}$ that of $\R^2$, we have the following basis of $\R^6$:
		\begin{equation} \label{eq:BasisR6}
			\begin{aligned}
				Y_1 & = E_1 \otimes e_1 & \ \ \ \  Y_2 & = E_1 \otimes e_2 & \ \ \ \  Y_3 & = E_1 \otimes e_3 \\
				Y_4 & = E_2 \otimes e_1 & \ \ \ \ Y_5 & = E_2 \otimes e_2 & \ \ \ \ Y_6 & = E_2 \otimes e_3.
			\end{aligned}
		\end{equation}
		We equip $\R^6$ with the flat metric, with respect to which $\{Y_1, \ldots, Y_6\}$ is an orthonormal basis. \\
		\indent The Lie group $\G = \SO(3) \times \SO(2)$ acts on $\R^6$ in the obvious way, with $\SO(3)$ acting on the domain, and $\SO(2)$ acting on the codomain.  Concretely, the $(\SO(3) \times \SO(2))$-action on $\R^6 = \Hom(\R^3, \R^2) \cong (\R^3)^* \otimes \R^2$ yields a map
		\begin{align*}
			\rho \colon \SO(3) \times \SO(2) & \to \SO(6) \\
			\rho(A, e^{it}) & =  \begin{bmatrix}
				\cos (t) A & -\sin (t)  A \\
				\sin (t) A & \cos (t)  A
			\end{bmatrix}\!.
		\end{align*}
		In this way, we obtain an $\SO(3) \times \SO(2)$-action on $S^5$. \\
		\indent We now compute the principal isotropy subgroup $\mathrm{H} \leq \SO(3) \times \SO(2)$.  For this, let $a \in \mathrm{Hom}(\R^3, \R^2)$.  Generically, $\Ker(a) \subset \R^3$ is $1$-dimensional.  Since $\SO(3)$ acts transitively on $\RP^2 = \{\text{1-dim subspaces of }\R^3\}$ with stabilizer $\mathrm{S}(\mathrm{O}(2) \times \mathrm{O}(1))$, we can assume that $\Ker(a) = \mathrm{span}(e_3)$ with residual ambiguity $\mathrm{S}(\mathrm{O}(2) \times \mathrm{O}(1)) \times \SO(2)$. \\
		\indent Now, let $P = \mathrm{span}(e_1, e_2) \subset \R^3$, and consider the restriction $\left.a\right|_P \colon P \to \R^2$.  By acting by the two $\SO(2)$-factors on the domain and codomain, we may diagonalize $a|_P$ so that
		\begin{equation*}
			a(e_1) = \lambda_1 E_1, \:\:\: a(e_2) = \lambda_2 E_2.
		\end{equation*}
		for $\lambda_1, \lambda_2 \neq 0$ with $\lambda_1^2 + \lambda_2^2 = |a|^2$.  Provided $\lambda_1 \neq \lambda_2,$ the residual ambiguity is $\mathbb{Z}_2,$ embedded in $\SO(3) \times \SO(2)$ as
		\begin{equation*}
			\mathrm{H} = \left\lbrace \left( \begin{bmatrix}
				\varepsilon \mathrm{Id}_2 & 0 \\
				0 & 1
			\end{bmatrix} , \varepsilon \mathrm{Id}_2 \right) \colon \varepsilon \in \lbrace 1, -1 \rbrace  \right\rbrace \cong \Z_2.
		\end{equation*}
		Thus, the principal orbits of the $\SO(3) \times \SO(2)$-action on $S^5$ are equivariantly diffeomorphic to $N := (\SO(3) \times \SO(2))/\Z_2$. 
		
		\subsubsection{Invariant metrics on $N$} We pause to consider the homogeneous space $N = (\SO(3) \times \SO(2))/\Z_2$.  Write the Maurer-Cartan forms on $\SO(3)$ and $\SO(2)$ as, respectively
		\begin{align*}
			& \begin{bmatrix}
				0 & \omega_3 & -\omega_2 \\
				-\omega_3 & 0 & \omega_1 \\
				\omega_2 & -\omega_1 & 0
			\end{bmatrix} & &
			\begin{bmatrix}
				0 & \omega_4 \\
				-\omega_4 & 0
			\end{bmatrix}
		\end{align*}
		Letting $\pi \colon \SO(3) \times \SO(2) \to N$ denote the coset projection, we see that every $\SO(3) \times \SO(2)$-invariant metric on $N$ pulls back to
		\begin{equation} \label{eq:InvMetricSO3SO2}
			\pi^*g_{b_1, \ldots, b_6} = \left( b_1\omega_1^2 + b_2 \omega_1 \circ \omega_2 + b_3 \omega_2^2\right) + \left( b_4\omega_3^2 + b_5 \omega_3 \circ \omega_4 + b_6 \omega_4^2 \right)
		\end{equation}
		for some constants $b_1, \ldots, b_6$. In particular, the normal metrics correspond to $b_1 = b_3 = b_4 > 0$ and $b_2 = b_5 = 0$ and $b_6 > 0$.
		
		\subsubsection{The isoparametric hypersurfaces $M_\theta \subset S^5$} 
		
		\indent We now study the principal $\SO(3) \times \SO(2)$-orbits in $S^5$.  To begin, note that the curve $h \colon [0, \frac{\pi}{4}] \to S^5$ given by
		$$h(\theta) = \cos(\theta)\,E_1 \otimes e_1 + \sin(\theta)\,E_2 \otimes e_2$$
		is a geodesic that intersects each $\SO(3) \times \SO(2)$-orbit exactly once, and the principal orbits correspond to $\theta \in (0, \frac{\pi}{4})$.  For $\theta \in (0,\frac{\pi}{4})$, we let $\varphi_\theta \colon N \to S^5$ denote the orbit map $\varphi((A, e^{it})\Z_2) = \rho(A, e^{it})(h(\theta))$, and let
		$$M_\theta := \varphi_\theta(N) \subset S^5$$
		denote the orbit of $h(\theta) \in S^5$. \\
		\indent For computations, we let $\widetilde{h} \colon [0, \frac{\pi}{4}] \to \SO(6)$ denote the following lift of $h$:
		\begin{equation*}
			\widetilde{h}(\theta) = \begin{bmatrix}
				\cos \theta & 0 & 0 & 0 & -\sin \theta & 0 \\
				0 & 1 & 0 & 0 & 0 & 0 \\
				0 & 0 & 1 & 0 & 0 & 0 \\
				0 & 0 & 0 & 1 & 0 & 0 \\
				\sin \theta & 0 & 0 & 0 & \cos \theta & 0 \\
				0 & 0 & 0 & 0 & 0 & 1.
			\end{bmatrix}
		\end{equation*}
		As in $\S$\ref{subsec:MovingFrames}, we let $\psi_\theta \colon \SO(3) \times \SO(2) \to \SO(6)$ denote the map
		$$\psi_\theta(A, e^{it}) = \rho(A, e^{it}) \circ \widetilde{h}(\theta).$$
		Altogether, we have a commutative diagram
		\begin{equation*}
			\begin{tikzcd}
				\SO(3) \times \SO(2) \arrow[r, "\psi_\theta"] \arrow[d, "\pi"'] & \SO(6) \arrow[d, "\mathbf{x}"] \\
				N \arrow[r, "\varphi_\theta"']  & S^5           
			\end{tikzcd}
		\end{equation*}
		Elements $(\mathbf{x}, \mathbf{e}_1, \ldots, \mathbf{e}_5)$ of $\SO(3) \times \SO(2) \leq \SO(6)$ may be viewed as adapted $\Z_2$-frames on $M_\theta$.  Geometrically, at a point $p \in M_\theta$, we have that $\mathbf{x} = p$ is the position vector and $\mathbf{e}_4$ is the normal vector to $M_\theta \subset S^5$ at $p$. \\
		
		\indent Working on $\SO(3) \times \SO(2)$ and suppressing pullbacks via $\psi_\theta \colon \SO(3) \times \SO(2) \to \SO(6)$ from the notation, an application of equation (\ref{eq:EtaAlphaBeta-Relation}) yields:
		\begin{equation*}
			\begin{aligned}
				\alpha_1 & = - \cos(\theta)\, \omega_3 + \sin(\theta)\, \omega_4, \\
				\alpha_2 & = \cos(\theta)\, \omega_2, \\
				\alpha_3 & = \sin(\theta)\, \omega_3 - \cos (\theta)\, \omega_4, \\
				\alpha_4 & = 0, \\
				\alpha_5 & = -\sin(\theta)\, \omega_1.
			\end{aligned}
		\end{equation*}
		Therefore, the induced metric on $M_\theta \subset S^5$ has
\begin{align*}
b_1 & = \sin^2\!\theta & b_2 & = 0 & b_3 & = \cos^2\!\theta & b_4 & = 1 & b_5 & = -4\cos \theta \sin \theta & b_6 & = 1.
\end{align*}  
    % \begin{equation*}
    %     \pi^*g_\theta = \left(\sin^2\theta\right)
    %     \omega_{1}^{2} + \left( \cos^2
    %     \theta \right)  \omega_{{2}}^{2}  + {\omega_{{3}}}^{2}-4\,\cos  \theta \sin \theta \,\omega_{3}  \omega_{4}+\omega_{4}^{2}.
    % \end{equation*}
		Moreover, the symmetric $2$-tensor given by $\SFF = \alpha_1 \circ \beta_{41} + \alpha_2 \circ \beta_{42} + \alpha_3 \circ \beta_{43} + \alpha_5 \circ \beta_{45}$ descends to $N$ as the second fundamental form of $M_\theta \subset S^5$.  Using (\ref{eq:EtaAlphaBeta-Relation}), one computes the mean curvature to be
		$$H(M_\theta) =  \frac{4 \sin(8\theta)}{\cos(8\theta) - 1},$$
		which vanishes when $\theta = \frac{\pi}{8}$.
		
		\subsubsection{The minimal hypersurface $M \subset S^5$} From now on, we fix $\theta = \frac{\pi}{8}$ and let $M := M_{\frac{\pi}{8}}$ denote the minimal isoparametric hypersurface in the $M_\theta$ family.  The induced metric $g_M$ on $M$ pulls back as
		\begin{equation*}
			\pi^*g_M = \frac{1}{4}(2-\sqrt{2})\, 
			\omega_{1}^{2} + \frac{1}{4}(2+\sqrt{2})\,  \omega_{{2}}^{2}  + {\omega_{{3}}}^{2} - \sqrt{2} \,\omega_{3}  \omega_{4}+\omega_{4}^{2}.
		\end{equation*}
		The Cartan-M\"{u}nzner polynomial of the isoparametric foliation is the quartic $F \colon \R^6 \to \R$ given by
		\begin{equation*}
			F(a) = 8 \, \det (a a^t) - \lvert a \rvert^4, 
		\end{equation*}
		where we are thinking of $a \in \R^6$ as an element of $\operatorname{Hom}(\R^3, \R^2)$.  In particular, $M = F^{-1}(0) \cap S^5$.
		
		\subsection{Decomposition of $L^2(M)$} 
		
		\subsubsection{$\SO(3) \times \SO(2)$ representation theory}
		
		The $\SO(3)$-action on $\C^3$ induces complex $\SO(3)$ representations on each $\mathcal{H}_p := \Sym^p_0(\C^3)$ for $p \geq 0$.  It is well-known that each $\mathcal{H}_p$ is irreducible, and that conversely, every irreducible complex $\SO(3)$-representation has this form.  Now, let $\C_q$ for $q \in \Z$ denote the complex irreducible $\SO(2)$-representation given by $e^{i\theta} \cdot z := e^{qi\theta} z$.  Then the complex irreducible $(\SO(3) \times \SO(2))$-representations are all of the form
		\begin{equation*}
			V_{p,q} := \mathcal{H}_p \otimes \C_q,
		\end{equation*}
		for $p,q \in \Z$ with $p \geq 0$.  Via the isomorphism $\C_q \cong \Sym^q(\C_1)$, we may view elements of $\C_q$ as homogeneous $q$th-degree polynomials. \\
		\indent Let $\{v_1, v_2, v_3\}$ be the standard basis of $\C^3$, and let $\{z_i z_j \colon 1 \leq i < j \leq 3\}$ be the induced basis of $\Sym^2(\C^3)$.  A computation shows that the following are highest weight vectors in their respective representations:
		\begin{align*}
			v_1 + iv_2 & \in V_{1,0} & z_1^2 - z_2^2 + 2iz_1z_2 & \in V_{2,0} \\
			v_1 + iv_2 & \in V_{1,1}.  & &
		\end{align*}
		Moreover, $\dim_{\C} (V_{0,2}) = 1$, so every non-zero vector in $V_{0,2}$ is a highest weight vector.
		Accordingly, we define the following projection maps $\pi_{p,q} \colon V_{p,q} \to \C$:
		\begin{equation} \label{eq:ProjMapsSO3SO2}
			\begin{aligned}
				\textstyle \pi_{1,0}(\sum z_j v_j) & = z_1 + i z_2 & \ \ \ \ \  \textstyle \pi_{2,0}( \sum c_{ij} z_i z_j ) & = c_{11} - c_{22} + 2ic_{12} \\
				\textstyle  \pi_{1,1}(\sum z_j v_j) & = z_1 + i z_2 & \ \ \ \ \  \pi_{0,2}(z) & = z
			\end{aligned}
		\end{equation}		
		\subsubsection{The multiplicity formula (Step 1).} By Frobenius Reciprocity (Proposition \ref{prop:FrobRecGen}), there is an $(\SO(3) \times \SO(2))$-invariant decomposition
		$$L^2(M;\C) = \bigoplus_{p \geq 0} \bigoplus_{q \in \Z} V_{p,q}^{\oplus m(p,q)}$$
		where the multiplicity of $V_{p,q}$ is
		$$m(p,q) = \dim_{\C}\{ v \in V_{p, q} \colon h \cdot v = v, \ \forall h \in \Z_2\}.$$
		
		\begin{prop} \label{prop:MultiplicitySO3SO2} We have
			\begin{equation*}
				\begin{aligned}
					m(p,q) &= p+1 \:\:\:\: \text{if} \:\:\:\: p \equiv q \, (2), \\
					m(p,q) &= p \:\:\:\: \text{if} \:\:\:\: p \equiv q +1 \, (2).
				\end{aligned}
			\end{equation*}
		\end{prop}
		
		\begin{proof}
			Let $(x,y,z)$ be the standard basis for $\C^3$ and let $w$ be a basis for $\C_q.$ Then the non-trivial element $\varepsilon \in \mathbb{Z}_2$ acts as
			\begin{equation*}
				\begin{aligned}
					\varepsilon \cdot (x,y,z) = (-x,-y,z), \:\:\: \varepsilon \cdot w = (-1)^q w.
				\end{aligned}
			\end{equation*}
			The action of $\varepsilon$ on $\mathbb{C}^3$ induces an action on $\mathcal{H}_p.$ When $q$ is even, we have that $m(p,q)$ is equal to the dimension of the $+1$ eigenspace of this action, while when $q$ is odd, $m(p,q)$ is the dimension of the $-1$ eigenspace.
			
			The monomials $x^ay^bz^c$ with $a+b+c = p$ furnish a basis for $\mathrm{Sym}^p(\C^3).$ The element $\varepsilon$ acts on these monomials via $\varepsilon \cdot x^ay^bz^c = (-1)^{a+b} x^ay^b z^c.$ If we let $E^{\pm}_p$ be the $\pm 1$ eigenspace for this action of $\varepsilon$ on $\mathrm{Sym}^p(\C^3),$ we have
			\begin{equation*}
				\begin{aligned}
					\dim E^{+}_p &= \sum_{i=0}^{\lfloor p/2 \rfloor} (2i+1) = \left( \lfloor p/2 \rfloor +1 \right)^2, \\
					\dim E^{-}_p &= \tfrac{1}{2}(p+1)(p+2) - \dim E^{+}_p = \tfrac{1}{2}(p+1)(p+2) - \left( \lfloor p/2 \rfloor +1 \right)^2.
				\end{aligned}
			\end{equation*}
			Since the contraction $\Sym^p(\C^3) \to \Sym^{p-2}(\C^3)$ is surjective with kernel $\mathcal{H}_p$, there is a decomposition $\Sym^p(\C^3) \cong \mathcal{H}_p \oplus \Sym^{p-2}(\C^3)$, and the action of $\varepsilon$ preserves this splitting.  Consequently, the dimensions of the eigenspaces of the $\varepsilon$-action on $\mathcal{H}_p$ are $\dim E^{\pm}_{p} - \dim E^{\pm}_{p-2}$, and
			\begin{equation*}
				\begin{aligned}
					\dim E^{+}_p - \dim E^{+}_{p-2} &= 2 \lfloor p /2 \rfloor +1, \\
					\dim E^{-}_p - \dim E^{-}_{p-2} &= 2p - 2 \lfloor p /2 \rfloor.
				\end{aligned}
			\end{equation*}
			Therefore, if $p$ and $q$ are both even or both odd, $m(p,q) = p+1,$ while if only one of $p$ and $q$ is even, $m(p,q) = p.$
		\end{proof}
		
		\subsection{Explicit equivariant functions (Step 2a)} Before continuing, we set notation.  To begin, we identify $\C^3 \simeq \R^6$ via
		$$(x_{11} + i x_{12}, x_{21} + i x_{22}, x_{31} + i x_{32}) \mapsto (x_{11}, x_{12}, x_{21}, x_{22}, x_{31}, x_{32}).$$
		Letting $\{e_1, e_2, e_3\}$ and $\{E_1, E_2\}$ be the standard bases of $\R^3$ and $\R^2$, respectively, we define $Y_{ia} := e_i \otimes E_a \in \R^6 = \Hom(\R^3, \R^2)$.  Note that in terms of the basis (\ref{eq:BasisR6}), we have that $(Y_{11}, Y_{12}, Y_{21}, Y_{22}, Y_{31}, Y_{32}) = (Y_1, Y_4, Y_2, Y_5, Y_3, Y_6)$. \\
		\indent We use the following index notation. Elements $x_{ia}Y_{ia}$ in $\R^6 = \operatorname{Hom}(\R^3, \R^2)$ are often abbreviated as $(x_{ia})$, where $1 \leq i \leq 3,$ $1 \leq a \leq 2$.  We let $\delta_{ab}$ denote the Kronecker delta and let $\varepsilon_{ijk}$ denote the totally-skew tensor with $\varepsilon_{123} = 1$. \\
		
		\indent We now provide a sufficiently large supply of $(\SO(3) \times \SO(2))$-equivariant functions $M \to V_{p,q}$.  First, we have the position vector $\mathbf{x}$ and normal vector $\mathbf{n}$, i.e.:
		\begin{align*}
			\mathbf{x} \colon M & \to \R^6 \cong V_{1,1} & \mathbf{n} \colon M & \to \R^6 \cong V_{1,1} \\
			\mathbf{x} &  = (x_{ia}) = (x_{11}, x_{12}, x_{21}, x_{22}, x_{31}, x_{32}) & \mathbf{n} & = (n_{ia}) = \frac{1}{4}\left( \frac{\partial F}{\partial x_{ia}}\right)\! = \frac{1}{4}\nabla F,
		\end{align*}
		where $F \colon \R^6 \to \R$ is the (quartic) Cartan-M\"{u}nzner polynomial. Next, we have the quadratic functions
		\begin{align*}
			\mathbf{k} \colon M & \to \C^3 \cong V_{1,0} & \mathbf{b} \colon M & \to \Sym^2(\C^3) \cong V_{2,0} \\
			\mathbf{k} & = (k_i) = \left( \tfrac{1}{2}\varepsilon_{ijk} (x_{j1} x_{k2} -  x_{j2} x_{k1}) \right) & \mathbf{b} & = (b_{ij}) = (x_{ia}x_{ja} - \tfrac{1}{3} x_{ka} x_{ka} \delta_{ij})
		\end{align*}
		and 
		\begin{align*}
			\mathbf{s} \colon M & \to (\Sym^2_\R)_0(\C_1) \cong \Sym^2(V_{0,1}) = V_{0,2} \\
			\mathbf{s} & = (s_{ab}) = (x_{ia} x_{ib} - \tfrac{1}{2} x_{ic} x_{ic} \delta_{ab}),
		\end{align*}
		where we have used the isomorphism from $\C_2$ to $(\mathrm{Sym}^2_\R)_0(\C_1)$ given by
		\begin{equation*}
			x + i y \mapsto \begin{bmatrix}
				x & y \\
				y & - x
			\end{bmatrix}\!.
		\end{equation*}
		Finally, we have the quartic function
		\begin{align*}
			\mathbf{r} \colon M & \to \Sym^2(\C^3) \cong V_{2,0} \\
			\mathbf{r} & = (r_{ij}) := 4\left(\left( x_{i1} n_{j2} + x_{j1} n_{i2} \right) - \left( x_{i2} n_{j1} + x_{j2} n_{i1} \right)\right)\!.
		\end{align*}
				
		\begin{rmk} The components of the position vector $\mathbf{x}$ and normal vector $\mathbf{n}$ are eigenfunctions with eigenvalues $-4$ and $-12$ respectively.   The components of $\mathbf{x} \wedge \mathbf{n}$ are eigenfunctions with eigenvalue $-16$.  In fact, the components of $\mathbf{x} \wedge \mathbf{n}$ take values in
			\begin{equation*}
				\Lambda^2 V_{1,1} / \left( \mathfrak{so}(3) \oplus \mathfrak{so}(2) \right) \cong V_{1,2} \oplus V_{2,0},
			\end{equation*}
			thereby yielding two families of eigenfunctions with eigenvalue $-16$.
		\end{rmk}

		\subsection{Decomposition of isotypical summands (Step 2b)} 
		
		Recalling the maps $\pi_{p,q} \colon V_{p,q} \to \C$ defined in (\ref{eq:ProjMapsSO3SO2}), we define the following six functions highest weight functions $M \to \C$:
		\begin{align*}
			\zeta & := \pi_{1,1} \circ \mathbf{x} & \rho & := i\,\pi_{2,0} \circ \mathbf{r} & \kappa & := \pi_{1,0} \circ \mathbf{k} \\
			\nu & := 4\pi_{1,1} \circ \mathbf{n} & \beta & := \pi_{2,0} \circ \mathbf{b} &  \sigma & : = -i \,\pi_{0,2} \circ \mathbf{s}
		\end{align*}
		The factor of $i = \sqrt{-1}$ in the definitions of $\rho$ and $\sigma$ is included purely to simplify certain formulas in $\S$\ref{subsec:DotProductSO3SO2}.  Explicitly,
		\begin{align*}
			\zeta & = (x_{11} - x_{22}) + i \left( x_{21} + x_{12} \right) & \beta &= (b_{11} - b_{22}) + 2 ib_{12}, \\
			\nu & = 4[(n_{11} - n_{22}) + i \left( n_{21} + n_{12} \right)], &  \rho &=  - 2 r_{12} + i \left(r_{11} - r_{22} \right), \\
			\kappa &= k_1 + i k_2, &  \sigma &= s_{12} - i s_{22}.
		\end{align*}
					
			\indent Now, for fixed $p, q, r \in \Z$ with $p \geq 0$ and  $r \in \{0,1\}$, we consider the following $(\SO(3) \times \SO(2))$-modules:
			\begin{align*}
				W(2p+r,2q) & := \mathrm{span}_{\C}\lbrace g^\sharp(\sigma^q \kappa^{i+r} \beta^j \rho^k) \colon  i + 2j + 2k = 2p, \ k \in \{0,1\}, \ g \in \SO(3) \times \SO(2) \rbrace, \\
				W(2p+1+r, 2q+1) & := \mathrm{span}_{\C}\lbrace g^\sharp(\sigma^q \kappa^r  \zeta^k \nu^{1-k} \rho^i \beta^j)  \colon i+j = p, \ k \in \{0,1\}, \ g \in \SO(3) \times \SO(2) \rbrace.
			\end{align*}
			We prove:
			
			\begin{prop} \label{prop:Isotypical-SO3SO2} For $p,q \in \Z$ with $p \geq 0$, we have
				$$W(p,q) \cong V_{p,q}^{\oplus m(p,q)}.$$
				Consequently,
				$$L^2(M;\C) \cong \bigoplus_{p \geq 0, q \in \Z} W(2p, 2q) \oplus W(2p+1,2q) \oplus W(2p+1, 2q+1) \oplus W(2p+2, 2q+1).$$
			\end{prop}
			
			\noindent The proof of Proposition \ref{prop:Isotypical-SO3SO2} will follow from the following three lemmas.
			
			\begin{lem} \label{lem:LinInd1}
				For fixed $p \in \Z_{\geq 0}$, the set $\lbrace \kappa^i \beta^j \mid i + 2j = 2p \rbrace \cup \lbrace \kappa^i \beta^j \rho \mid i + 2j = 2p -2 \rbrace$ is linearly independent.
			\end{lem}
			
			\begin{proof}
				We argue by induction.  For $p = 0$, the set is clearly linearly independent.  Now, consider the following two points in $M \subset \R^6$:
				\begin{equation*}
					\begin{aligned}
						p_1 & = \textstyle \left( 0, \,\tfrac{1}{2}\sqrt {2-\sqrt {2}}, \,0, \,-\tfrac{1}{4}\left( -2+\sqrt {2} \right) 
						\sqrt {2\,\sqrt {2}+4}, \,0, \, \tfrac{1}{2}\,{2}^{3/4} \right), \\
						p_2 & = \textstyle \left( \tfrac{1}{2}\,\sqrt {\sqrt {2}-1},\, 0, \, \tfrac{1}{2}\,\sqrt {2},\,0,\, \tfrac{1}{2}\,\sqrt {\sqrt {2}-1},\,0 \right)\!.
					\end{aligned}
				\end{equation*}
				They satisfy
				\begin{align*}
					\kappa(p_1) & = -\frac{i}{4} \, {2}^{3/4}\sqrt {2-\sqrt {2}}, \:\:\: & \rho(p_1) & = -8\,i \left( \sqrt {2}-1 \right), \:\:\: & \beta(p_1) & = 0, \\
					\kappa(p_2) & = \frac{i}{4}\sqrt {2}\sqrt {\sqrt {2}-1}, \:\:\: & \rho(p_2) & = 4\,i \left( \sqrt {2}-1 \right), \:\:\: & \beta(p_2) & = 0.
				\end{align*}
				Suppose the set $\mathcal{S}_{2p} := \lbrace \kappa^i \beta^j \mid i + 2j = 2p \rbrace \cup \lbrace \kappa^i \beta^j \rho \mid i + 2j = 2p -2 \rbrace$ satisfies a linear relation, say
				\begin{equation} \label{eq:LinRel1}
					\sum_{i+2j = 2p} a_{i,j} \kappa^i \beta^j + \sum_{i+2j = 2p-2} b_{i,j} \kappa^i \beta^j \rho = 0.
				\end{equation}
				Evaluating at $p_1$ and $p_2$ and using $\beta(p_1) = \beta(p_2) = 0$, we find
				\begin{equation*}
					\begin{aligned}
						\kappa(p_1)^2 a_{2p,0} + \rho(p_1) b_{2p-2,0} &= 0, \\
						\kappa(p_2)^2 a_{2p,0} + \rho(p_2) b_{2p-2,0} &= 0,
					\end{aligned}
				\end{equation*}
				from which it follows that $a_{2p,0} = b_{2p-2,0} = 0.$ Therefore, the linear relation (\ref{eq:LinRel1}) has a factor of $\beta$, so that linear independence for $\mathcal{S}_{2p}$ follows from that of $\mathcal{S}_{2p-2}$.  We are done by induction.
			\end{proof}
			
			\begin{lem}
				For fixed $p \in \Z_{\geq 0}$, the set $\lbrace \zeta \rho^i \beta^j \mid i+j = p \rbrace \cup \lbrace \nu \rho^i \beta^j \mid i+j = p \rbrace$ is linearly independent.
			\end{lem}
			
			\begin{proof}
				We again argue by induction.  Consider the point $p_1 \in M \subset \R^6$ defined in the previous proof, and the point $p_3 \in M$ given by
				\begin{equation*}
					\begin{aligned}
						p_3 & = \textstyle \left( 0, \, \tfrac{1}{2} \left( 2-\sqrt {2} \right)^{1/2}, \,\tfrac{1}{2}  {2}^{3/4}, \,\tfrac{1}{2} \left( 2-\sqrt {2} \right)^{1/2},\, 0 , \,0 \right)\!.
					\end{aligned}
				\end{equation*}
				We have
				\begin{align*}
					\zeta(p_3) & = 0,  & \nu(p_3) & = -4i\sqrt {2}\sqrt {2-\sqrt {2}}, & \rho(p_3) & = 8i \!\left( \sqrt {2}-1 \right), & \beta(p_3) & = 0.
				\end{align*}
				Suppose the set $\mathcal{S}_{2p+1} := \lbrace \zeta \rho^i \beta^j \mid i+j = p \rbrace \cup \lbrace \nu \rho^i \beta^j \mid i+j = p \rbrace$ satisfies a linear relation, say
				\begin{equation} \label{eq:LinRel2}
					\sum_{i+j = p} a_{i,j} \zeta \rho^i \beta^j + \sum_{i+j=p} b_{i,j} \nu \rho^i \beta^j = 0.
				\end{equation}
				Evaluating this relation at $p_1$ and $p_3$, using $\beta(p_1) = \beta(p_3) = 0$, $\zeta(p_3) = 0$, and $\rho(p_1) \neq 0$, $\rho(p_3) \neq 0$, we find
				\begin{align*}
					a_{p,0} \zeta(p_1)  +  b_{p,0} \nu(p_1) & = 0, & b_{p,0} \nu(p_3) & = 0.
				\end{align*}
				Noting that $\zeta(p_1) = (\sqrt{2} - 1)^{1/2} \neq 0$, it follows that $a_{p,0} = b_{p,0} = 0.$ Therefore, the linear relation (\ref{eq:LinRel2}) has a factor of $\beta$, so that linear independence for $\mathcal{S}_{2p+1}$ follows from that of $\mathcal{S}_{2p-1}$. Linear independence of $\mathcal{S}_1$ follows from $\zeta(p_3) = 0,$ $\nu(p_3) \neq 0$.  This completes the induction.
			\end{proof}
			
			\begin{lem} \label{lem:BasisHWSO3SO2} Fix $p,q \in \Z$ with $p \geq 0$.  Let $I(p,q)$ be the set of highest weight functions with weight $(p,q)$.
				\begin{enumerate}
					\item The set $\lbrace \sigma^q \kappa^i \beta^j \mid i + 2j = 2p \rbrace \cup \lbrace \sigma^q \kappa^i \beta^j \rho \mid i + 2j = 2p -2 \rbrace$ is a basis of $I(2p,2q)$.
					\item The set $\lbrace \sigma^q \kappa^{i+1} \beta^j \mid i + 2j = 2p \rbrace \cup \lbrace \sigma^q \kappa^{i+1} \beta^j \rho \mid i + 2j = 2p -2 \rbrace$ is a basis of  $I(2p+1,2q)$.
					\item The set $\lbrace  \sigma^q \zeta \rho^i \beta^j  \mid i+j = p \rbrace \cup \lbrace \sigma^q \nu \rho^i \beta^j  \mid i+j = p \rbrace$ is a basis of $I(2p+1, 2q+1)$.
					\item The set $\lbrace \sigma^q \kappa  \zeta\rho^i \beta^j  \mid i+j = p \rbrace \cup \lbrace \sigma^q \kappa  \nu \rho^i \beta^j  \mid i+j = p \rbrace$ is a basis of  $I(2p+2, 2q+1)$.
				\end{enumerate}
			\end{lem}
			
			\begin{proof}
				By Proposition \ref{prop:SpanHighestWeight}(a) and Proposition \ref{prop:MultiplicitySO3SO2}, we have
				\begin{align*}
					\dim( I(2p,2q) ) & = m(2p,2q) = 2p+1 \\
					\dim( I(2p+1, 2q )) & = m(2p+1, 2q) = 2p+1.
				\end{align*}
				Now, by Lemma \ref{lem:LinInd1}, the set of $2p+1$ functions $\lbrace \kappa^i \beta^q \mid i + 2j = 2p \rbrace \cup \lbrace \kappa^i \beta^q \rho \mid i + 2j = 2p -2 \rbrace$ is linearly independent.  Since multiplication by $\sigma^q$ is an injective map, claims (1) and (2) follow.  An analogous argument establishes claims (3) and (4).
			\end{proof}

			\subsection{Computing $\Delta_M$ on isotypical summands (Step 3)} \label{subsec:DotProductSO3SO2} We now aim to compute the Laplacian $\Delta_M$ on each isotypical summand $W(p,q)$.  For this, it suffices to compute the Laplacians on the basis functions of $I(p,q)$ given in Lemma \ref{lem:BasisHWSO3SO2}. \\
			
			\indent We start by calculating the pairwise dot products of the gradients $\nabla \zeta$, $\nabla \nu$, etc., along with the Laplacians $\Delta \zeta$, $\Delta \nu$, etc.  The result is as follows:
			
			\begin{prop} \label{prop:DotProdutsSO3SO2} We have
				\begin{equation*}
					\begin{aligned}
						\nabla \kappa \cdot \nabla \beta &= 16 \kappa^3, & \nabla \kappa \cdot \nabla \rho &= -8 \kappa \rho, \\
						\nabla \kappa \cdot \nabla \sigma &= 0, & \nabla \rho \cdot \nabla \sigma &= 64 \beta \sigma + 256 \kappa^2 \sigma, \\
						\nabla \beta \cdot \nabla \sigma &= \tfrac{1}{2} \rho \sigma, & \nabla \beta \cdot \nabla \rho &= -8 \beta \rho, \\
						\nabla \kappa \cdot \nabla \kappa &= -\beta - 8 \kappa^2, & \nabla \sigma \cdot \nabla \sigma &= - 8 \sigma^2, \\
						\nabla \sigma \cdot \nabla \zeta &= -2 \sigma \zeta + \tfrac{1}{2} \sigma \nu, & \nabla \sigma \cdot \nabla \nu &= 8 \sigma \zeta - 6 \sigma \nu, \\
						\nabla \kappa \cdot \nabla \zeta &= -2 \kappa \zeta - \tfrac{1}{2} \kappa \nu, & \nabla \kappa \cdot \nabla \nu &= -8 \kappa \zeta - 6 \kappa \nu, \\
						\nabla \zeta \cdot \nabla \beta &= -4 \zeta \beta - \nu \beta + \tfrac{1}{2} \zeta \rho - \tfrac{1}{16} \nu \rho, & \nabla \nu \cdot \nabla \beta &= -16 \zeta \beta - 12 \nu \beta + 3 \zeta \rho, \\
						\nabla \zeta \cdot \nabla \rho &= 32 \zeta \beta -4 \zeta \rho, & \nabla \nu \cdot \nabla \rho &= -96 \nu \beta + 32 \zeta \rho -12 \nu \rho, \\
						\nabla \beta \cdot \nabla \beta &= -32 \beta \kappa^2 - 64 \kappa^4 - 8 \beta^2. & & 
					\end{aligned}
				\end{equation*}
				We also have the following Laplacian formulas:
				\begin{equation*}
					\begin{aligned}
						\Delta \zeta &= - 4 \zeta, & \Delta \nu &= -12 \nu, \\
						\Delta \beta &= - 8 \beta + 16 \kappa^2, & \Delta \rho &= - 16 \rho, \\
						\Delta \kappa &= -8 \kappa, & \Delta \sigma &= -8 \sigma.
					\end{aligned}
				\end{equation*}
			\end{prop}
			
			\begin{proof}
				These formulas are proven in the same manner as Proposition \ref{prop:DotProdGradSO(5)}. We use the moving frame approach of \S\ref{subsubsec:Differentiation} to express the desired dot product or Laplacian as an expression in an adapted frame $(\mathbf{x}, \mathbf{e}_1, \ldots, \mathbf{e}_5),$ which we restrict to a specific submanifold $S$ of $M$ (together with an adapted frame at each point of $S$) and compare with the basis for each space of highest weight functions with a given weight. The specific submanifold and adapted framing we use is the image of 
				\begin{equation*}
					(t,s) \mapsto \psi_{\pi / 8} ( \exp(sA), e^{it} ),
				\end{equation*}
				where $A \in \SO(3)$ is
				\begin{equation*}
					A = \frac{2}{\sqrt{5}} \left[ \begin {array}{ccc} 0 & 1 & \tfrac{1}{2}\\ \noalign{\medskip} -1 & 0 & 0 \\ \noalign{\medskip}-\tfrac{1}{2} & 0 & 0 \end {array} \right].
				\end{equation*}
				
			\end{proof}
			
			\indent Next, for a fixed value of $p$ and $r$, define functions
			\begin{equation*}
				\begin{aligned}
					f_j &= \sigma^r \beta^j \kappa^{2p-2j}, \:\:\: 0 \leq j \leq p, \\
					g_j &= \rho \sigma^r \beta^{j} \kappa^{2p-2j-2}, \:\:\: 0 \leq j \leq p-1.
				\end{aligned}
			\end{equation*}
			Using the formulas in Proposition \ref{prop:DotProdutsSO3SO2}, together with the Laplacian product rule, it is straightforward to compute:
			
			\begin{prop} \label{prop:LaplacianWEven} We have
				\begin{equation*}
					\begin{aligned}
						\Delta \!\left( f_j \right) & = jr\,g_{j-1} - 64\,j \left( j-1 \right) f_{j-2} - 16 j \left( 6j - 4p - 3 \right) f_{j-1} \\
						& + \left( 4p - 4j - 2 \right)  \left( j-p \right) f_{j+1}+ \left( -40\,j^2 + 64\,pj - 32p^2 - 8\,r^2 \right) f_j, \\
						\\
						\Delta\! \left( g_{{j}} \right) & = \left( -40\,j^2 + 64\,pj - 32\,p^2 - 8r^2 - 48j + 32p - 16 \right) g_j + \left( 4p - 4j - 6
						\right)  \left( j-p+1 \right) g_{j+1} \\
						& -16j \left( 6j - 4p + 1 \right) g_{j-1} - 64j \left( j-1 \right) g_{j-2} + 1024\,rj\,f_{j-1} \\
						& + 128 r \left( j+1 \right) f_{j+1}  + 512r \left( 2\,j+1 \right) f_j,
					\end{aligned}
				\end{equation*}
				and
				\begin{equation*}
					\begin{aligned}
						\Delta\! \left( \kappa f_j \right) & = \left( -40 j^2 + 8\left( 8 p + 4 \right) j - 32 p^2 - 8r^2 - 32p - 8 \right) 
						\kappa f_j
						\\ & + \left( -4j^2 + 800 \left( 4p + 1 \right) j - 8p \left( 2p + 1 \right)  \right) \kappa\,f_{j+1}+ \left( -64j^
						2 + 64 j \right) \kappa f_{j-2} \\
						& + \left( -96j^2 + 200 \left( 8p + 10 \right) j \right) \kappa f_{j-1}+ jr\,\kappa g_{j-1},
						\\
						\\
						\Delta\! \left( \kappa g_j \right) & = \left( 1024jr + 512 r
						\right) \kappa f_j + 128\left( jr + r \right) \kappa f_{j+1} \\
						& + 1024 rj\,\kappa f_{j-1} + \left( -40 j^2 + 64pj - 32p^2 -
						8 r^2 - 16j - 8 \right) \kappa g_j \\
						& + \left( -4j^2 + 8pj - 4p^2 - 6j + 6p - 2 \right) \kappa g_{j+1}+ \left( -64 j^2 + 64\,j \right) \kappa g_{j-2} \\
						& + \left( -96j^2 + 64pj + 16j \right) \kappa g_{j-1}.
					\end{aligned}
				\end{equation*}
			\end{prop}
			The formulas in Proposition \ref{prop:LaplacianWEven} yield matrix expressions of the Laplacians on the isotypical components $W(2p,2r)$ and $W(2p+1,2r)$. \\
			
			\indent Finally, for fixed values of $p$ and $r$, define functions 
			\begin{equation*}
				\begin{aligned}
					a_j &= \sigma^r \zeta \rho^{p-j} \beta^j \:\:\: 0 \leq j \leq p, \\
					b_j &=  \sigma^r \nu \rho^{p-j} \beta^j \:\:\: 0 \leq j \leq p.
				\end{aligned}
			\end{equation*}
			Again using Proposition \ref{prop:DotProdutsSO3SO2}, together with the Laplacian product rule, it is straightforward to compute:
			
			\begin{prop} \label{prop:LaplacianWOdd} We have
				\begin{equation*}
					\begin{aligned}
						\Delta \!\left( a_j \right) &= 64 \left( -2p + 2j + 1 \right) 
						\left( j-p \right) a_{j+1} \\
						& + \left( -8j^2 + 16pj - 16jr - 16p^2 + 16pr - 8r^2 - 4j - 8p - 4r - 4 \right) a_j \\
						& + \tfrac{1}{2}\,j
						\left( 2j + 2r + 1 \right) a_{j-1} - \tfrac{1}{16}\,j \left( j-1 \right) a_{j
							-2} - 256\, \left( j-p+1 \right)  \left( j-p \right) b_{j+2} \\
						& + 32r \left( j-p \right) b_{j+1}+ \left( -2j^2 - j + r \right) b_j
						-\tfrac{1}{8}j\,b_{j-1}
					\end{aligned}
				\end{equation*}
				and
				\begin{equation*}
					\begin{aligned}
						\Delta \!\left( b_j \right) & = -4096 \left( j-p+1 \right)  \left(j-p \right) a_{j+2} + 512 \left( j-p \right)  \left( 2j - 2p + r + 2
						\right) a_{j+1} \\
						& + \left( -32j^2 - 128jr + 128pr - 80j + 64p - 16r \right) a_j + 2j \left( 4j+1 \right) a_{j-1} \\
						& -2048 \left( j-p+1 \right)  \left( j-p \right) b_{j+2}+ \left( 64p - 64
						j \right)  \left( 2j - 2p - 4r - 1 \right) b_{j+1} \\
						& + \left( -24j^2 + 16pj + 16jr - 16p^2 - 16pr - 8r^2 + 4j - 24p - 12r - 12 \right) b_j \\
						& -\tfrac{1}{2}j \left( 2j - 2r - 1 \right) b_{j-1} - \tfrac{1}{16}j\left( j-1 \right) b_{j-2},
					\end{aligned}
				\end{equation*}
				and
				\begin{equation*}
					\begin{aligned}
						\Delta \!\left( \kappa a_j \right) & = 64 \left( -2p + 2j+1
						\right)  \left( j-p \right) \kappa a_{j+1} \\
						& + \left( -8j^2 + 16pj - 16jr - 16p^2 + 16pr - 8r^2 + 4j - 24p - 4r - 16 \right) 
						\kappa a_j \\
						& + \tfrac{1}{2}j \left( 2j + 2r + 3 \right) \kappa a_{j-1} - \tfrac{1}{16} j \left( j-1 \right) \kappa a_{j-2} - 256 \left( j-p+1 \right)  \left( j-p \right) \kappa b_{j+2} \\
						& + 32 r \left( j-p \right) \kappa b_{j+1} + \left( -2 j^2 - 3j + r - 1 \right) \kappa b_j - \tfrac{1}{8}j\,\kappa b_{j-1}
					\end{aligned}
				\end{equation*}
				and
				\begin{equation*}
					\begin{aligned}
						\Delta\! \left( \kappa b_j \right) & = -4096 \left( j-p+1 \right) 
						\left( j-p \right) \kappa a_{j+2} + 512 \left( j-p \right) 
						\left( 2j - 2p + r + 2 \right) \kappa a_{j+1} \\
						& + \left( -32j^2 -128jr + 128pr - 112j + 64p - 16r - 16 \right) \kappa a_j + 2j
						\left( 4j + 5 \right) \kappa a_{j-1} \\
						& -2048 \left( j-p+1 \right) 
						\left( j-p \right) \kappa b_{j+2}+ \left( 64p - 64 j \right) 
						\left( 2j - 2p - 4r - 1 \right) \kappa b_{j+1} \\
						& + \left( -24j^2 + 16 pj + 16jr - 16 p^2 - 16pr - 8r^2 - 4j - 40p - 12r - 32
						\right) \kappa b_j \\
						& - \tfrac{1}{2}  \left( 2j - 2r + 1 \right) j\,\kappa b_{
							j-1} - \tfrac{1}{16} j \left( j-1 \right) \kappa b_{j-2}.
					\end{aligned}
				\end{equation*}
			\end{prop}
			The formulas in Proposition \ref{prop:LaplacianWOdd} yield matrix expressions of the Laplacians on the isotypical components $W(2p,2r+1)$ and $W(2p+1,2r+1)$.
			
			\subsection{The index and nullity (Step 4)} 
			\indent For $p,q \in \Z$ with $p \geq 0$, let $Y(p,q) \subset L^2(M;\R)$ denote the real subspace spanned by the real and imaginary parts of the highest weight functions in $I(p,q)$ (recall Lemma \ref{lem:BasisHWSO3SO2}), and let $\Delta_{p,q} := \Delta|_{Y(p,q)}$ denote the Laplacian of $g_M$ restricted to $Y(p,q)$. \\
			
			\indent We are almost ready to compute the Morse index and nullity of $M \subset S^5$.  For this, we need an exhaustive list of the eigenvalues (and their multiplicities) of $-\Delta$ that are less than $16$.  Now, Proposition \ref{prop:LaplacianWEven} (respectively, Proposition \ref{prop:LaplacianWOdd}) yields the matrix of $\Delta_{p, 2q}$ (resp., $\Delta_{p,2q+1}$) with respect to the basis consisting of the real and imaginary parts of highest weight functions in $I(p,2q)$ (resp., $I(p,2q+1)$). Thus, it remains to understand which $Y(p,q)$ house eigenvalues less than $16$. \\
			\indent For this, we will apply Theorem \ref{thm:MainBound} to the homogeneous space $N = (\SO(3) \times \SO(2))/\Z_2$.  In prepraration, we let $g_1 = \varphi^*_{\pi/8}g_M$ denote the pullback of the minimal isoparametric metric to $N$, and note that $\Delta_{p,q}$ is a self-adjoint endomorphism with respect to the $L^2$ inner product on $Y(p,q)$ given by $(f_1, f_2) = \int_M f_1 f_2\,\vol_{g_N}$.  Next, we let $g_2 = g_\mathrm{nor}$ denote the normal homogeneous metric on $N$ given by (\ref{eq:InvMetricSO3SO2}) with $b_1 = b_3 = b_4 = b_6 = 1$ and $b_2 = b_5 = 0$, and let $L_{p,q}$ denote the Laplacian of $g_{\mathrm{nor}}$ restricted to $Y(p,q)$.  By standard theory \cite{yamaguchi79},
			$$L_{p,q} = -\mu(p,q)\,\mathrm{Id},$$
			where $\mu(p,q)$ is the eigenvalue of the Casimir of $\SO(3) \times \SO(2)$ on $Y(p,q)$.  Moreover,
			$$\mu(p,q) = p(p+1) + q^2.$$
			We now have:
			
			\begin{prop} \label{prop:Bounds-SO3SO2} Let $\lambda$ be an eigenvalue of $-\Delta_{p,q}$.  Then:
				\begin{align*}
					\left( 2 - \sqrt{2} \right) \left( p(p+1) + q^2 \right) & \leq \lambda \leq \left( 4 + 2 \sqrt{2} \right) \left( p(p+1) + q^2 \right)\!.
				\end{align*}
			\end{prop}
			
			\begin{proof} As before, we let $\pi \colon \SO(3) \times \SO(2) \to N$ denote the coset projection.  Recall that $g_{\mathrm{nor}}$ and $g_M$ pull back to $\SO(3) \times \SO(2)$ as
				\begin{align*}
					\pi^*g_{\mathrm{nor}} & =  \omega_1^2 + \omega_2^2 +  \omega_3^2 + \omega_4^2  \\
					\pi^*g_M & = \textstyle \frac{1}{4}(2-\sqrt{2})\, 
					\omega_{1}^{2} + \frac{1}{4}(2+\sqrt{2})\,  \omega_{{2}}^{2}  + {\omega_{{3}}}^{2} - \sqrt{2} \,\omega_{3}  \omega_{4}+\omega_{4}^{2}.
				\end{align*}
				In particular, the smallest and largest eigenvalues of $g_M$ with respect to $g_{\mathrm{nor}}$ are $\frac{1}{4}(2-\sqrt{2})$ and $\frac{1}{2}(2 + \sqrt{2})$.  Taking reciprocals, the smallest and largest eigenvalues of the dual metric $g_M^{-1}$ with respect to $g_{\mathrm{nor}}^{-1}$ are $r_{\min} = 2-\sqrt{2}$ and $r_{\max} = 4 + 2 \sqrt{2}$.  Applying Theorem \ref{thm:MainBound} yields the result.
			\end{proof}
			
			\indent We can now understand which eigenspaces of $-\Delta$ have eigenvalue at most $16$.  Indeed, by Proposition \ref{prop:Bounds-SO3SO2}, we need only consider the operators $-\Delta_{p,q}$ for $(p,q)$ satisfying
			$$\left( 2 - \sqrt{2} \right) \left( p(p+1) + q^2\right) \leq 16.$$
			That is, we need only consider the following values of $(p,q)$:
			\begin{equation*}
				\begin{aligned}
					& (0,0),\,(0,\pm 2),\,(0,\pm 4), \\
					& (1,0),\,(1,\pm 1),\,(1,\pm 2),\,(1,\pm 3),\,(1,\pm 4),\,(1,\pm 5), \\
					& (2,0),\,(2,\pm 1),\,(2,\pm 2),\,(2,\pm 3),\,(2,\pm 4), \\
					& (3,0),\,(3,\pm 1),\,(3,\pm 2),\,(3,\pm 3),\,(3,\pm 4), \\
					& (4,0),\,(4,\pm 1),\,(4,\pm 2).
				\end{aligned}
			\end{equation*}
			Using the matrix representations of $-\Delta_{p,q}$ provided by Propositions \ref{prop:LaplacianWEven} and \ref{prop:LaplacianWOdd}, we find that the following are all the eigenvalues of $-\Delta$ less than or equal to $16$:
			
			$$\begin{tabular}{| l | l | c |} \hline
				Eigenvalue of $-\Delta$ & $(p,q)$ & Multiplicity \\ \hline \hline
				$\lambda_1 = 0$ & $(0,0)$ & 1 \\ \hline
				$\lambda_2 = 4$ & $(1,1) \oplus (1,-1)$ & 6 \\ \hline
				$\lambda_3 = 8$ & $(0,2) \oplus (0,-2) \oplus (1,0)$ & 5 \\ \hline
				$\lambda_4 = 20 - 4 \sqrt{7} \approx 9.42$ & $(2,0)$ & 5 \\ \hline
				$\lambda_5 = 12$ & $(1,1) \oplus (1,-1)$ & 6 \\ \hline
				$\lambda_6 = 24 - 4\sqrt{5} \approx 15.06$ & $(1,3) \oplus (1,-3) \oplus (2,1) \oplus (2,-1)$ & 16 \\ \hline \hline
				$\lambda_7 = 16$ & $(1,2) \oplus (1,-2) \oplus (2,0)$ & 11 \\ \hline
			\end{tabular}$$ \\
			Note that the multiplicity of an eigenvalue corresponding to $V_{p,q}$ is given by $\dim(V_{p,q}) = 2p+1$.  The above table now yields:
			
			\begin{thm} The index of the $(\SO(3) \times \SO(2))$-invariant minimal isoparametric hypersurface in $S^5$ is 39. Furthermore, the nullity is 11 and every Jacobi field is rotational.
			\end{thm}
			
			\begin{proof} Summing the dimensions of the eigenspaces of $-\Delta$ having eigenvalue less than $16$, we obtain
				\begin{align*}
					\text{Ind}(M) &  = 1 + 6 + 5 + 5 + 6 + 16 = 39.
				\end{align*}
				Moreover, since $\dim(\SO(6)) - \dim(\SO(3) \times \SO(2)) = 11$, the space of rotational Jacobi fields is $11$-dimensional.  Since the nullity of $M$ is also equal to $11$, we deduce that every Jacobi field is rotational.
			\end{proof}
			
			\begin{thm} The Laplace spectrum of the minimal isoparametric metric on $(\SO(3) \times \SO(2))/\Z_2$ contains eigenvalues that are not expressible in radicals.
			\end{thm}
			
			\begin{proof} For $(p,q) = (4,4)$, the Laplacian  acts as
				\begin{equation*}
					\Delta_{4,4} =
					\left[ \begin {array}{ccccc}
					-160&-12&0&0&0\\
					80&-72&-2&2&0\\
					-128&-32&-64&0&4\\ 
					1024&256&0&-112&-2\\ 
					2048&3072&512&16&-72
					\end {array} \right].
				\end{equation*}
				The characteristic polynomial of this matrix is 
				\begin{equation*}
					x^5 + 480\,x^4 + 87264\,x^3 + 7409920\,x^2 + 287600640\,x+
					3933241344
				\end{equation*}
				which has Galois group $\mathrm{S}_5,$ so its roots cannot be expressed in radicals.
			\end{proof}

			\section*{Appendix: Representation theory}

			Let $\G$ be a compact, connected Lie group with Lie algebra $\mathfrak{g}$ and let $(V, \langle \cdot, \cdot \rangle)$ be a 
			hermitian vector space with a unitary $\G$-representation, i.e. a group homomorphism  $\varphi: \G \rightarrow \mathrm{Aut}(V)$. 
			Any representation $V$ of a compact Lie group $\G$ is the direct sum of irreducible summands. Combining isomorphic summands  we obtain
			the decomposition of $V$ into {\it isotypical components}, i.e. sums of  irreducible summands all isomorphic to a fixed irreducible $\G$-representation.
			The {\it multiplicity} of an irreducible $\G$-representation $W$ in $V$ is then the number of summands in the isotypical component defined by $W$.
			It is given by the dimension of $\mathrm{Hom}_{\G}(V, W)$, the space of $\G$-equivariant homomorphisms from $V$ to $W$.
			
			We recall that any representation $(V, \varphi)$ of a Lie group $\G$  defines a representation of its Lie algebra $\mathfrak{g}$, 
			i.e. we have a Lie algebra homomorphism $\varphi_* : \mathfrak{g} \rightarrow \mathrm{End}(V)$, where  $\varphi_*$ is the differential of 
			the representation $\varphi$. This is a 1-1-correspondence in case of simply connected Lie groups.
			A representation $(V, \varphi)$ of a Lie group $\G$ is also called a {\it $\G$-module}. Then $\G$-invariant
			subspaces of $V$ are $\G$-submodules and $\G$-equivariant homomorphisms  are $\G$-module morphisms.

			We fix a maximal torus $T \subset \G$ with Lie algebra $\mathfrak{t} \subset \mathfrak{g}$. 
			Then the restriction to the maximal torus defines a representation of $T$. Since $T$ is abelian it decomposes into a sum of $1$-dimensional 
			complex representations, the simultaneous eigenspaces of the $T$ action on $V$. The {\it weight vectors} of the
			$\G$-module $V$ are the simultaneous eigenvectors of this restricted $T$-representation. 
			The  {\it weights} of  $V$ are the corresponding eigenvalues considered as $1$-forms on the Lie algebra $\mathfrak{t}$. 
			They span a lattice in the Lie algebra $\mathfrak{t}^*$.
			More explicitly, a vector $v\in V$ is a weight vector of weight $\omega$ if and only if $\varphi_*(X)v = 2\pi i \omega(X) v$ 
			for all $X \in \mathfrak{t}$.  A {\it weight space} is the subspace of weight vectors for a fixed weight. These are
			exactly the isotypical components of the $T$-representation on $V$. 
			
			As an example we have the  {\it roots}  of the Lie algebra $\mathfrak{g}$, which are the non-zero weights of the 
			adjoint representation of $\G$ on the complexification $\mathfrak{g}^\C$ of its Lie algebra. Fixing a basis of simple roots distinguishes a Weyl chamber and  defines a partial 
			ordering on the weight lattice. For the classical groups, in particular for $\SO(n)$, we will use the basis of
			simple roots as given in Bourbaki \cite{bourbaki89}. In the latter case, identifying $\mathfrak{t}$ with $\mathbb{R}^{[\frac{n}{2}]}$, we obtain
			the lexicographic order on $\mathbb{R}^{[\frac{n}{2}]}$.
			
			By the {\it Theorem of the highest weight} any irreducible unitary $\G$-module has a {\it highest weight},
			characterizing the  $\G$-module up to equivalence, i.e. up to $\G$-module  isomorphisms. The corresponding weight space is $1$-dimensional, i.e.
			the corresponding highest weight vector is unique up to scalar multiples. In this sense  irreducible unitary $\G$-module are
			parametrized by their highest weights and we will write $V= V_\rho$ if the $V$ has highest weight $\rho$.
			The set of all equivalence classes of irreducible, finite-dimensional unitary $\G$-modules  is denoted as  $\hat \G$ 
			and, identifying  irreducible unitary $G$-module with their highest weight, we will  also write $\rho \in \hat \G$.
			
			The {\it fundamental weights} $\omega_1, \ldots, \omega_n$ are defined as a basis of $\mathfrak{t}$ dual  to the set of coroots associated to the simple roots. 
			That is, the fundamental weights are defined by the condition:
			$$
			2 \, \frac{\langle \omega_i, \alpha_j \rangle}{\langle \alpha_j, \alpha_j \rangle} \,=\, \delta_{ij} ,
			$$
			where $\alpha_1, \ldots, \alpha_n$ are the simple roots. The fundamental representations, i.e. the corresponding $\G$-modules of highest
			weights $\omega_i, i=1, \ldots, n$, generate the representation ring of $\G$.
			
			Finally, we mention the following facts about weights that we will need in this work: 
			
			\begin{enumerate}
				\item 
				Every $\G$-module has an orthogonal basis of weight vectors.
				\item
				The span of the $\G$-orbit of a highest weight vector is irreducible.
				\item
				The tensor product $V_\rho \otimes V_\mu$ contains  $V_{\rho + \mu}$ as $\G$-submodule of multiplicity one. 
				The corresponding highest weight vector is exactly the tensor product of the two highest weight vectors of the two factors.
			\end{enumerate}
		
		\bibliography{IndexQuarticreferences}
			
		\end{document}